\documentclass{amsart}

\usepackage{calc}
\usepackage{tikz}
\usepackage{pgfplots}
\usetikzlibrary{shapes}
\usepackage{amsmath,amssymb}
 \usepackage[foot]{amsaddr}
\let\oldtocsection=\tocsection
 
\let\oldtocsubsection=\tocsubsection 
 
\let\oldtocsubsubsection=\tocsubsubsection
 
\renewcommand{\tocsection}[2]{\vspace{0.5em}\hspace{0em}\oldtocsection{#1}{#2}}
\renewcommand{\tocsubsection}[2]{\vspace{0.5em}\hspace{1em}\oldtocsubsection{#1}{#2}}
\renewcommand{\tocsubsubsection}[2]{\vspace{0.5em}\hspace{2em}\oldtocsubsubsection{#1}{#2}}
\usepackage{graphicx,cancel,xcolor,hyperref,comment,graphicx,geometry}
 
\usepackage{tikz}
\usepackage{graphicx,url,etoolbox}
\usepackage{lipsum}
\usepackage{dsfont}
\usepackage{subfigure}
\usetikzlibrary{hobby,patterns}
\usepackage{fancyhdr}
\usepackage{lastpage}

\newtheorem{theoreme}{Theorem}[section]
\newtheorem{pro}[theoreme]{Proposition}
\newtheorem{lemma}[theoreme]{Lemma}
\newtheorem{rem}[theoreme]{Remark}

\theoremstyle{definition}

\numberwithin{equation}{section}

 \renewenvironment{proof}{{\bfseries \noindent Proof.}}{\demo}
\newcommand\xqed[1]{%
  \leavevmode\unskip\penalty9999 \hbox{}\nobreak\hfill
  \quad\hbox{#1}}
\newcommand\demo{\xqed{$\square$}}

\hypersetup{bookmarks, colorlinks, urlcolor=blue, citecolor=red, linkcolor=blue, hyperfigures, pagebackref,
    pdfcreator=LaTeX, breaklinks=true, pdfpagelayout=SinglePage, bookmarksopen=true,bookmarksopenlevel=2}

\usepackage{comment}
\def\u2{\u^2}
\def\u3{\u^3}
\def\u4{\u^4}
\def\u5{\u^5}
\def\y1{\y^1}
\def\y2{\y^2}
\def\y3{\y^3}
\def\y4{\y^4}
\def\y5{\y^5}

\def\R{\mathbb R}

\def\HH{\mathcal H}

\def\AA{\mathcal A}

\def\la {{\lambda}}

\newcommand {\nc}   {\newcommand}
\nc {\be}   {\begin{equation}} \nc {\ee}   {\end{equation}} \nc
{\beq}  {\begin{eqnarray}} \nc {\eeq}  {\end{eqnarray}} \nc {\beqs}
{\begin{eqnarray*}} \nc {\eeqs} {\end{eqnarray*}}
\def\edc{\end{document}}

\providecommand{\abs}[1]{\lvert#1\rvert}%absolute value
\usepackage{tikz}
\usetikzlibrary{decorations.pathmorphing,patterns,scopes,intersections,calc}
\usepackage{caption}

\newcounter{dummy} 
\numberwithin{dummy}{section}
\newtheorem{Theorem}[dummy]{Theorem}

\newtheorem{defi}[dummy]{Definition}

\newtheorem{Remark}[dummy]{Remark}

\numberwithin{equation}{section}

\begin{document}
\title[\fontsize{7}{9}\selectfont  ]{Stability of pizoelectric beam with magnetic effect under (Coleman or Pipkin)-Gurtin thermal law}
\author{Mohammad Akil$^{1}$}
%\author{Haidar Badawi$^{1,2}$}
%\author{Serge Nicaise$^{1}$}
\address{$^1$Universit\'e Polytechnique  Hauts-de-France, C\'ERAMATHS/DEMAV, Le Mont Houy 59313 Valenciennes Cedex 9-France}
%\address{$^2$ Aix-Marseille University, I2M, Marseille-France.}
\email{mohammad.akil@uphf.fr}

%\date{}
%%%%%%%%%%%%%%%%%%%%%%%%%%%%%%%%%%%%%%%%%%%%%%%%%%%%%%%%%%%%
%Keyword
%%%%%%%%%%%%%%%%%%%%%%%%%%%%%%%%%%%%%%%%%%%%%%%%%%%%%%%%%%%%
%\keywords{}
%%%%%%%%%%%%%%%%%%%%%%%%%%%%%%%%%%%%%%%%%%%%%%%%%%%%%%%%%%%%
%Abstract
%%%%%%%%%%%%%%%%%%%%%%%%%%%%%%%%%%%%%%%%%%%%%%%%%%%%%%%%%%%%

%\pagenumbering{roman}
%\maketitle
%\tableofcontents
%\clearpage
%\pagenumbering{arabic}
%\setcounter{page}{1}
%\setcounter{equation}[section]
\setcounter{equation}{0}
%\abstractname{.}

\begin{abstract}
In this paper, we investigate the stabilization of a system of piezoelectric beams under (Coleman or Pipkin )-Gurtin thermal law with magnetic effect. First, we study the Piezoelectric-Coleman-Gurtin system and we obtain an exponential stability result. Next, we consider the Piezoelectric-Gurtin-Pipkin system and we establish a polynomial energy decay rate of type $t^{-1}$ . 

\end{abstract}
\maketitle
\pagenumbering{roman}
\maketitle
\tableofcontents
%\clearpage
\pagenumbering{arabic}
\setcounter{page}{1}
%\newpage
\section{Introduction} 
\noindent \subsection{Piezoelectric Beam} It is known, since the 19th century that materials such as quartz, Rochelle salt and barium titanate under pressure produces electric charge\slash voltage, this phenomenon is called the direct piezoelectric effect and was discovered by brothers Pierre and Jacques Curie in 1880. This same materials, when subjected to an electric field, produce proportional geometric tension. Such a  phenomenon is known as the converse piezoelectric effect and was discovered by Gabriel Lippmann in 1881.\\
Morris and Ozer, proposed a piezoelectric beam model with a magnetic effect, based on the Euler-Bernoulli and Rayleigh  beam theory for small displacement (the same equations for the model are obtained if Midlin-Timoshenko small displacement assumptions bare used ), they considered an elastic beam covered by a piezoelectric material on its upper and lower surfaces, isolated by the edges and connected to a external electrical circuit to feed charge to the electrodes. As the voltage is prescribed at the electrodes, the following Lagrangian was considered 
\begin{equation}\label{Lagrangian}
	\mathcal{L}=\int_0^T\left[\bold{K-(P+E)+B+W}\right]dt,
\end{equation}
where $\bold{K}$, $\bold{P+E}$, $\bold{B}$ and $\bold{W}$ represent the (mechanical) kinetic energy, total stored energy, magnetic energy (electrical kinetic) of the beam and the work done by external forces, respectively. For a beam length $L$ to thickness $h$ and considering $v=v(x,t)$, $w=w(x,t)$ and $p=p(x,t)$ as functions that represent the longitudinal displacement of the center line, transverse displacement of the beam and the total load of the electric displacement along the transverse direction at each point $x$, respectively. So, one can assume that 
\begin{equation}\label{newL}
	\begin{array}{cc}
	\displaystyle
	\bold{P+E}=\frac{h}{2}\int_0^L\left[\alpha\left(v_x^2+\frac{h^2}{12}w_{xx}^2-2\gamma\beta v_xp_x+\beta p_x^2\right)\right]dx, & \displaystyle
	\bold{B}=\frac{\mu h}{2}\int_0^Lp_t^2dx,\\[0.1in]
	\displaystyle
	\bold{K}=\frac{\rho h}{2}\int_0^L\left[v_t^2+\left(\frac{h^2}{12}+1\right)\omega_t^2\right],& \displaystyle 
	\bold{W}=-\int_0^Lp_xV(t)dx,
	\end{array}
\end{equation}
where $V(t)$ is the voltage applied at the electrode. From Hamilton's principle for admissible displacement variations displacement variations $\left\{v,w,p\right\}$ of $L$ the zero and observing that the only external force acting on the beam is the voltage at the electrodes (the bending equation is decoupled) see \cite{Morris-Ozer2013,Morris-Ozer2014}, they got the system 
\begin{equation}\label{piezo}
\begin{array}{c}
\rho v_{tt}-\alpha v_{xx}+\gamma \beta p_{xx}=0,\\
\mu p_{tt}-\beta p_{xx}+\gamma \beta v_{xx}=0,
\end{array}	
\end{equation}
where $\rho, \alpha, \gamma, \mu$ and $\beta$ denote the mass density, elastic stiffness, piezoelectric coefficient, magnetic permeability, water resistance coefficient of the beam and the prescribed voltage on electrodes of beam respectively, and in addition, the relationship 
\begin{equation}\label{alpha1}
\alpha=\alpha_1+\gamma^2\beta.	
\end{equation}
They assumed that the beam is fixed at $x=0$ and free at $x=L$, and thus they got (from modelling) the following boundary conditions  
\begin{equation}\label{bc}
\begin{array}{c}
v(0,t)=\alpha v_x(L,t)-\gamma \beta p_x(L,t)=0,\\
p(0,t)=\beta p_x(L,t)-\gamma \beta v_x(L,t)=- \displaystyle{\frac{V(t)}{h}}.	
\end{array}	
\end{equation}
Then, the authors considered $V(t)=kp_t(L,t)$ (electrical feedback controller) in \eqref{bc} and established strong stabilization for almost all system parameters and exponential stability for system parameters in a null measure set. In \cite{Ramos2018} Ramos et al. inserted a dissipative term $\delta v_t$ in the first equation of \eqref{piezo} , where $\alpha>0$ is a constant and considered the following boundary condition 
\begin{equation}\label{Ramos-bc}
\begin{array}{c}
v(0,t)=\alpha v_x(L,t)-\gamma \beta p_x(L,t)=0,\\
p(0,t)=\beta p_x(L,t)-\gamma \beta v_x(L,t)=0.	
\end{array}	
\end{equation}
The authors showed, by using energy method, that the system's energy decays exponentially. This means that the friction term and the magnetic effect work together in order to uniformly stabilize the system. In \cite{Abdelaziz1}, the authors considered a one-dimensional dissipative system of piezoelectric beams with magnetic effect and localized damping. They proved that the system is exponential stable using a damping mechanism acting only on one component and on a small part of the beam. In \cite{Abdelaziz2}, the authors considered a one-dimensional piezoelectric beams with magnetic effect damped with a weakly nonlinear feedback in the presence of a nonlinear delay term.They established an energy decay rate under appropriate assumptions on the weight of the delay. In \cite{AnLiuKong}, the authors studied the stability of a piezoelectric beams with magnetic effects of fractional derivative type and with/ without thermal effects of fourrier's law, they obtained an exponential stability by taking two boundary fractional dampings and additional thermal effect. 

\subsection{(Coleman or Pipkin)-Gurtin thermal law} 
The theory of heat conduction under various non-Fourier heat flux laws has been developed since the 1940s. Let ${\bold{q}}$ be the heat flux vector. According to the Gurtin-Pipkin theory \cite{Gurtin1968}, the linearized constitutive equation of $q$ is   
\begin{equation}\label{GP}
\bold{q}(t)=-\int_0^{\infty}g(s)\theta_x(t-s)ds,	
\end{equation}
where $g$ is the heat conductivity relaxation kernel. The presence of convolution term
in \eqref{GP} entails finite propagation speed of heat conduction, and consequently, the equation is of hyperbolic type. Note that \eqref{GP} reduces to the classical Fourier law when $g$ is the Dirac mass at zero. Furthermore, if we take $g$ as a prototype kernel 
\begin{equation}\label{gGP}
g(t)=e^{-kt},\quad k>0,	
\end{equation}
and differentiate \eqref{GP} with respect to t, we can (formally) arrive at the so-called Cattaneo-Fourier law.. 
\begin{equation}\label{CF}
\bold{q}_t(t)+k\bold{q}=-\theta_x(t).	
\end{equation}
On the other hand, when the heat conduction is due to the Coleman-Gurtin theory \cite{Coleman1967}, the heat flux $\bold{q}$ depends on both the past history and the instantaneous of the gradient of temperature: 
\begin{equation}\label{H-CG}
\bold{q}(t)=-\beta\theta_x(x,t)-\int_0^{\infty}g(s)\theta_x(x,t-s)ds, 	
\end{equation}
where $\beta>0$ is the instantaneous diffusivity coefficient. The analysis on stabilization and controllability of the heat conduction equations under non-Fourier heat flux laws can be found in \cite{article1,article2,article3,article4} and references therein.\\
In \cite{Zhang2014-heat}, Q. Zhang studied the stability of an interaction system comprised of a wave equation and a heat equation with memory. An exponential stability of the interaction system is obtained when the hereditary heat conduction is of Gurtin-Pipkin type and she showed the lack of uniform decay of the interaction system when the heat conduction law is of Coleman-Gurtin type. Later, in \cite{Deloro-Lassi}, the authors studied the asymptotic behaviour of solutions of a one-dimensional coupled wave-heat system with Coleman-Gurtin thermal law. They proved an optimal polynomial decay rate of type $t^{-2}$. In \cite{DELLORO2021148}, the author studied the stability of Bresse and Timoshenko systems with hyperbolic heat conduction. First, he studied the Bresse-Gurtin-Pipkin system, providing a necessary and sufficient condition for the exponential stability and the optimal polynomial decay rate when the condition  is violated, also he studied the Timoshenko-Gurtin-Pipkin system and he find the optimal polynomial decay rate.

\subsection{Description of the model} Based on the description mentioned above piezoelectric beam and heat law, we design and propose to study the stability of the following system 
% Piezoelectric materials have the inherent characteristic of piezolectricity. We consider the following system 
\begin{equation}\label{Pizo}\tag{${\rm P_{CG}}$}
\left\{\begin{array}{l}
\rho u_{tt}-\alpha u_{xx}+\gamma \beta y_{xx}+\delta w_x=0,\quad (x,t)\in (0,L)\times (0,\infty),\\
\mu y_{tt}-\beta y_{xx}+\gamma \beta u_{xx}=0,\quad (x,t)\in (0,L)\times (0,\infty),\\
\displaystyle 
w_t-c(1-m)w_{xx}-c m\int_{0}^{\infty}g(s)w_{xx}(x,t-s)ds+\delta u_{xt}=0,\\
u(0,t)=y(0,t)=w(0,t)=w(L,t)=0,\\
\alpha u_x(L,t)-\gamma \beta y_x(L,t)=\beta y_x(L,t)-\gamma \beta u_x(L,t)=0.
\end{array}
\right.
\end{equation}
The convolution kernel $g:[0,\infty[\rightarrow [0,\infty[$ is a convex integrable function (thus non-increasing and vanishing at infinity) of unit total mass, taking the explicit form 
$$
g(s)=\int_s^{\infty}\sigma(r)dr,\quad s\geq 0,
$$
where $\sigma:(0,\infty)\to [0,\infty)$, called memory kernel, satisfying the following conditions
\begin{equation}\label{CONDITION-H}\tag{${\rm H}$}
\left\{\begin{array}{l}
\sigma \in L^1((0,\infty))\cap C^1((0,\infty))\quad \text{with}\quad \displaystyle{\int_0^{\infty}\sigma(r)dr=g(0)>0},\ \sigma(0)=\lim_{s\to 0}\sigma(s)<\infty,\\[0.1in]
\sigma\ \text{satisfies the Dafermos condition}\ \sigma'(s)\leq -d_\sigma \sigma(s).
\end{array}
\right.
\end{equation}	
\noindent Finally, we impose the initial conditions of the form 
\begin{equation}\label{IC-Pizo}
\left\{\begin{array}{ll}
\left(u(x,0),u_t(x,0),y(x,0),y_t(x,0)\right)=(u_0(x),u_1(x),y_0(x),y_1(x)),& x\in (0,L),\\
\left(w(x,0),w(x,-s)\right)=(w_0(x),\phi_0(x,s)),& x\in (0,L),\ s>0,
\end{array}
\right.
\end{equation}
where $u_0,u_1,y_0,y_1$ are assigned data and $c>0$. In particular, $\phi_0$ accounts for the so called initial past history of $w$.  In the model \eqref{Pizo}, $m\in [0,1]$ is a fixed parameter and the temperatures obey the parabolic hyperbolic law introduced by B.D Coleman and M.E. Gurtin in \cite{Coleman1967}. The limit cases:\\
$\bullet$ $m=0$ corresponds to the Pizoelectric-Fourrier law defined by: 
\begin{equation}\label{Pizo-F}\tag{${\rm P_F}$}
\left\{\begin{array}{l}
\rho u_{tt}-\alpha u_{xx}+\gamma \beta y_{xx}+\delta w_x=0,\quad (x,t)\in (0,L)\times (0,\infty),\\
\mu y_{tt}-\beta y_{xx}+\gamma \beta u_{xx}=0,\quad (x,t)\in (0,L)\times (0,\infty),\\
\displaystyle 
w_t-cw_{xx}+\delta u_{xt}=0,\\
u(0,t)=y(0,t)=w(0,t)=w(L,t)=0,\\
\alpha u_x(L,t)-\gamma \beta y_x(L,t)=\beta y_x(L,t)-\gamma \beta u_x(L,t)=0.
\end{array}
\right.
\end{equation}
$\bullet$ $m=1$ corresponds to the Pizoelectric-Gurtin-Pikin  law defined by: 
\begin{equation}\label{Pizo-GP}\tag{${\rm P_{GP}}$}
\left\{\begin{array}{l}
\rho u_{tt}-\alpha u_{xx}+\gamma \beta y_{xx}+\delta w_x=0,\quad (x,t)\in (0,L)\times (0,\infty),\\
\mu y_{tt}-\beta y_{xx}+\gamma \beta u_{xx}=0,\quad (x,t)\in (0,L)\times (0,\infty),\\
\displaystyle 
w_t-c\int_{0}^{\infty}g(s)w_{xx}(x,t-s)ds+\delta u_{xt}=0,\\
u(0,t)=y(0,t)=w(0,t)=w(L,t)=0,\\
\alpha u_x(L,t)-\gamma \beta y_x(L,t)=\beta y_x(L,t)-\gamma \beta u_x(L,t)=0.
\end{array}
\right.
\end{equation}
It is important to note that since $y(x,t)=\int_0^xD(\zeta,t)d\zeta$, where $D(\zeta,t)$ represents the electric displacement in the direction $z$, then $y(0,t)=0$ and still $y(L,t)$ may not be zero, because the boundary condition $p(L,t)=0$ does not represent the fixation of the beam on both sides. In fact, the fixation is due to the boundary condition $u(0,t)=u(L,t)=0$, where $u$ is the tranverse displacement of the beam. \\
This paper is organised as follow: In the first part we study the well-posedness of system \eqref{Pizo}. Next, we prove that the piezoelectric system with Coleman-Gurtin law is exponentially stable. In the last part, we consider the system of piezoelectric beam under Gurtin-Pipkin thermal law and we establish polynomial stability of type $t^{-1}$. 
%%%%%%%%%%%%%%%%%%%%%%%%%%%%%%%%%%%%%%%%%%%%%%
%%%%%%% Well-posedness 
%%%%%%%%%%%%%%%%%%%%%%%%%%%%%%%%%%%%%%%%%%%%%%
\section{Well-Posedness}
\noindent We start by introducing some notations and spaces used in this paper. First, we define
$$
H_L^1(0,L)=\left\{u\in H^1(0,L);\quad u(0)=0\right\}.
$$ 
It is easy to check that the space $H_L^1(0,L)$ is a (complex) Hilbert space over $\mathbb{C}$ equipped with  the inner product 
$$
\left<u,u^1\right>_{H_L^1(0,L)}=\left<u_x,u^1_x\right>_{L^2(0,L)}.
$$
We also introduce the \textit{memory space}  $W$, defined by 
$$
W=L^2_{\sigma}(\mathbb{R}^+;H_0^1(0,L)),
$$ 
of $H_0^1(0,L)-$valued functions on $(0,\infty)$ which are square integrable with respect to the measure $\sigma(s)ds$, endowed with the inner product 
$$
(\eta_1,\eta_2)_W:=cm\int_0^L\int_0^{\infty}\sigma(s)\eta_{1x}\overline{\eta_{2x}}dsdx,\quad \forall \eta_1,\eta_2\in W. 
$$
%Now, we introduce the infinitesimal generator of the  right-translation semigroup on $W$, defined by 
%$$
%T\eta=-\eta_s,\quad D(T)=\left\{\eta\in W,\ \eta_s\in W,\ \lim_{s\to 0}\int_0^L\abs{\eta_x(x,s)}^2dx=0\right\},
%$$
%where $\eta_s$ denotes the weak derivative of $\eta$ with respect to the variable $s>0$.
\noindent We reformulate the \eqref{Pizo} using the history framework of Dafermos \cite{Dafermos1970}. To this end, for $s>0$, we consider the auxiliary function 
$$
\eta (x,s)=\int_0^sw(x,t-r)dr,\quad x\in (0,L),\ s>0 
$$ 
and we rewrite \eqref{Pizo} in the form 
\begin{equation}\label{Pizo1}\tag{${\rm P_{CG1}}$}
\left\{\begin{array}{l}
\rho u_{tt}-\alpha u_{xx}+\gamma \beta y_{xx}+\delta w_x=0,\quad (x,t)\in (0,L)\times (0,\infty),\\
\mu y_{tt}-\beta y_{xx}+\gamma \beta u_{xx}=0,\quad (x,t)\in (0,L)\times (0,\infty),\\
\displaystyle 
w_t-c(1-m)w_{xx}-c m\int_{0}^{\infty}\sigma(s)\eta_{xx}(s)ds+\delta u_{xt}=0,\\
\eta_t+\eta_s-w=
0,\\
u(0,t)=y(0,t)=w(0,t)=w(L,t)=0,\\
\alpha u_x(L,t)-\gamma \beta y_x(L,t)=\beta y_x(L,t)-\gamma \beta u_x(L,t)=0,\quad (x,t)\in (0,L)\times (0,\infty)\times (0,\infty),
\end{array}
\right.
\end{equation}
with the initial conditions \eqref{IC-Pizo}. The energy of system \eqref{Pizo1} is given by 
\begin{equation}\label{Energy-Pizo}
E_m(t)=E_{m,1}(t)+E_{m,2}(t)+E_{m,3}(t).	
\end{equation}
where 
$$
\begin{array}{c}
\displaystyle 
E_{m,1}(t)=\frac{1}{2}\int_0^L\left(\rho \abs{u_t}^2+\alpha_1\abs{u_x}^2\right)dx,\ E_{m,2}(t)=\frac{1}{2}\int_0^L\left(\mu\abs{y_t}^2+\beta\abs{\gamma u_x-y_x}^2\right)dx\\ 
\text{and}\displaystyle \quad E_{m,3}(t)=\frac{1}{2}\int_0^L\abs{\omega}^2dx+\frac{cm}{2}\int_0^{L}\int_0^\infty\sigma(s)\abs{\eta_x}^2dsdx.
\end{array}
$$
\begin{lemma}\label{Energy}
Let $U=(u,u_t,y,y_t,w,\eta)$ be a regular solution of system \eqref{Pizo1}. Then, the energy $E_m(t)$ satisfies the following estimation 
\begin{equation}\label{dEnergy-m}
\frac{d}{dt}E_m(t)=c(m-1)\int_0^L\abs{\omega_x}^2dx+cm\int_0^{+\infty}\int_0^L	\sigma'(s)\abs{\eta_x}^2dxds. 
\end{equation}
\end{lemma}
%%%%%%%%%%%%%%%%%%%%%%%%%%%%%%%%
%%%%%%%%%%%%%%%%%%%%%%%%%%%%%%%%
\begin{proof}
Multiplying the first and the second  equation of \eqref{Pizo1} by $\overline{u_t}$ and $\overline{y_t}$ respectively, integrating by parts over $(0,L)$, we get 
\begin{equation}\label{dene1}
\frac{1}{2}\frac{d}{dt}\left(\rho\int_0^L\abs{u_t}^2dx+\alpha\int_0^L \abs{u_x}^2dx\right)-\gamma \beta \Re\left(\int_0^Ly_x\overline{u_{xt}}\right)+\delta \Re\left(\int_0^L\omega_x\overline{u_t}dx\right)=0
\end{equation}
and
\begin{equation}\label{dene2}
\frac{1}{2}\frac{d}{dt}\left(\mu\int_0^L\abs{y_t}^2dx+\beta\int_0^L \abs{y_x}^2dx\right)-\gamma \beta \Re\left(\int_0^Lu_x\overline{y_{xt}}\right)=0.
\end{equation}
Adding \eqref{dene1} and \eqref{dene2}, and using the fact that $\alpha=\alpha_1+\gamma^2\beta$, we get 
\begin{equation}\label{dene3}
\frac{1}{2}\frac{d}{dt}\left(\int_0^L(\rho\abs{u_t}^2+\alpha_1\abs{u_x}^2+\mu\abs{y_t}^2+\beta \abs{\gamma u_x-y_x}^2)dx\right)+\delta \int_0^L\omega_x\overline{u_t}dx=0.
\end{equation}
Now, multiplying the third equation of \eqref{Pizo1} by $\overline{w}$, integrating by parts over $(0,L)$, we get 
\begin{equation}\label{dene3}
\frac{1}{2}\frac{d}{dt}\int_0^L\abs{\omega}^2dx+c(1-m)\int_0^L\abs{\omega_x}^2dx+\Re\left(cm\int_0^L\int_0^{\infty}\sigma(s)\eta_x(s)\overline{\omega}_xdsdx\right)-\delta\int_0^Lu_t\overline{w}_xdx=0. 
\end{equation}
Differentiating the fourth equation with respect to $x$, we obtain 
\begin{equation}\label{dene4}
\eta_{xt}+\eta_{xs}-w_x=0. 	
\end{equation}
Multiplying \eqref{dene4} by $cm\sigma(s)\overline{\eta}_x$, integrating over $(0,L)\times (0,\infty)$, we get 
\begin{equation}\label{dene5}
\frac{1}{2}\frac{d}{dt}cm\int_0^{L}\int_0^{\infty}\sigma(s)\abs{\eta_x}^2dsdx-\frac{cm}{2}\int_0^{L}\int_0^{\infty}\sigma'(s)\abs{\eta_x}^2dsdx=\Re\left(cm\int_0^L\int_0^{\infty}\sigma(s)\eta_x\overline{w}_xdsdx\right)	
\end{equation}
Inserting \eqref{dene5} in \eqref{dene4}, we get 
\begin{equation}\label{dene6}
\frac{d}{dt}E_{m,3}(t)-\delta\int_0^Lu_t\overline{w}_xdx=c(m-1)\int_0^L\abs{\omega_x}^2dx+\frac{cm}{2}\int_0^{L}\int_0^\infty\sigma'(s)\abs{\eta_x}^2dsdx.
\end{equation}
\noindent Finally, adding \eqref{dene3} and \eqref{dene6}, we get 
the desired result. The proof has been completed. 
\end{proof}
%%%%%%%%%%%%%%%%%%%%%%%%%%%%%%%%%%%%%%%%%%%%%%%%%%%%%%%%

\noindent Now, we define the The Hilbert space $\mathcal{H}$ (Energy space) by 
$$
\mathcal{H}:=\left(H_L^1(0,L)\times L^2(0,L)\right)^2\times L^2(0,L)\times W,  
$$
 equipped with the following inner product 
$$
\left<U_1,U_2\right>_{\mathcal{H}}=\int_0^L\left[\rho\,v_1\overline{v_2}+\alpha_1u
_{1,x}\overline{u_{2,x}}+\mu\,z_1\overline{z_2}+\beta\left(\gamma u_{1,x}-y_{1,x}\right)\overline{\left(\gamma u_{1,x}-y_{1,x}\right)}+w_1\overline{w_2}\right]dx+(\eta_1,\eta_2)_W,
$$
where $U_i=(u_i,v_i,y_i,z_i,w_i)\in \mathcal{H}$, $i=1,2$.
\begin{rem}
By using the fact that $\alpha=\alpha_1+\gamma^2\beta$, the boundary conditions  at $L$ should be replaced by the Neumann conditions $u_x(L)=y_x(L)=0$.	
\end{rem}
\noindent By introducing the state $U=(u,v,y,z,\omega,\eta(\cdot,s))^{\top}$, system \eqref{Pizo1} can be written as the following first order evolution equation 
\begin{equation}\label{evolution}
	U_t={\mathcal{A}}_mU,\quad U(0)=U_0,
\end{equation}
where ${\mathcal{A}}_m:D({\mathcal{A}}_m)\subset \mathcal{H}\longrightarrow \mathcal{H}$ is an unbounded linear operator defined  by 
$$
\mathcal{A}_{m}\begin{pmatrix}
u\\ v\\ y\\ z\\ \omega\\ \eta	
\end{pmatrix}
=\begin{pmatrix}
v\\ \displaystyle \frac{1}{\rho}\left(\alpha u_{xx}-\gamma \beta y_{xx}-\delta \omega_x\right)\\ z\\ \displaystyle \frac{1}{\mu}\left(\beta y_{xx}-\gamma\beta u_{xx}\right)\\  \displaystyle{c\Lambda^m_{xx}-\delta v_x}\\
-\eta_s+\omega 
\end{pmatrix}
$$ 
and 
$$
D({\mathcal{A}}_m)=\left\{
\begin{array}{l}
U:=(u,v,y,z,w,\eta)\in \mathcal{H};\ v,z\in H_L^1(0,L);\ u,y\in H^2(0,L)\cap H_L^1(0,L),\ w\in H_0^1(0,L),\\
\displaystyle 
\Lambda^m_x\in H^1(0,L),\ \eta_s\in W,\ \eta(\cdot,0)=0\quad \text{and}\quad u_x(L)=y_x(L)=0.
\end{array}
\right\}
$$
where $\displaystyle{\Lambda^m=(1-m)\omega+m\int_0^{\infty}\sigma(s)\eta(s)ds}$ and $U_0=(u_0,u_1,y_0,y_1,\omega_0,\eta_0)^{\top}\in \mathcal{H}$ with $\eta_0=\displaystyle{\int_0^s\phi_0(x,r)dr}$ for $x\in (0,L)$ and $s>0$.
\begin{pro}\label{m-dissipative}
Under the hypothesis \eqref{CONDITION-H}, the unbounded linear operator ${\mathcal{A}}_m$ is m-dissipative in the energy space $\mathcal{H}$.  	
\end{pro}
\begin{proof}
For all $U=(u,v,y,z,w,\eta(\cdot,s))^{\top}\in D({\mathcal{A}}_m)$, by using conditions $(H1)$,  $(H2)$ and the fact that $m\in (0,1)$, it's easy to see that 
\begin{equation}\label{dissipative}
\Re\left(\left<{\mathcal{A}}_mU,U\right>_{\mathcal{H}}\right)=c(m-1)\int_0^L\abs{\omega_x}^2dx+\frac{cm}{2}\int_0^{L}\int_0^\infty	\sigma'(s)\abs{\eta_x}^2dsdx\leq 0,
\end{equation}
which implies that ${\mathcal{A}}_m$ is dissipative. Now, let us prove that ${\mathcal{A}}_m$ is maximal. For this aim, let $F=(f^1,f^2,f^3,f^4,f^5,f^6(\cdot,s))^{\top}\in \mathcal{H}$, we want to find $U=(u,v,y,z,w,\eta(\cdot,s))^{\top}\in D({\mathcal{A}}_m)$ unique solution of 
\begin{equation}\label{max1}
-{\mathcal{A}}_m U=F.	
\end{equation}
Equivalently, we have the following system 
\begin{eqnarray}
-v&=&f^1,\label{max2}\\
-\alpha u_{xx}+\gamma \beta y_{xx}+\delta \omega_x &=&\rho f^2,\label{max3}\\
-z&=&f^3,\label{max4}\\
-\beta y_{xx}+\gamma \beta u_{xx}&=&\mu f^4,\label{max5}\\
-c\Lambda^m_{xx}+\delta v_x&=&f^5,\label{max6}\\
\eta_s-w &=&f^6.\label{max7}	
\end{eqnarray} 
Thanks to \eqref{max2} and \eqref{max4}, it follows that, $v,z\in H_L^1(0,L)$ and 
\begin{equation}\label{max8}
v=-f_1\quad \text{and}\quad z=-f_3.
\end{equation}
From \eqref{max7}, we obtain 
\begin{equation}\label{1max8}
\eta(\cdot,s)=s w(x)+\int_0^s f^6(x,\tau)d\tau.
\end{equation}
Then, from \eqref{max2} and \eqref{max6}, we obtain 
\begin{equation}\label{2max9}
\Lambda^m_{xx}=-c^{-1}\left(f^5+\delta f^1_x\right),	
\end{equation}
it yields, 
\begin{equation}\label{3max9}
\Lambda^m(x)=-c^{-1}\int_0^x\int_0^{x_1} \left(f^5+\delta f^1_{x_2}\right)	dx_2dx_1+c^{-1}\frac{x}{L}\int_0^L\int_0^{x_1}\left(f^5+\delta f^1_{x_2}\right)dx_2dx_1.
\end{equation}
Now, using the definition of $\Lambda^m$, \eqref{1max8} and \eqref{3max9}, we get 
\begin{equation}\label{4max9}
\begin{array}{l}
\displaystyle 
w=\frac{1}{\tilde{m}}\left[\Lambda^m(x)-m\int_0^\infty\sigma(s)\int_0^sf^6(\tau)d\tau ds\right],
\end{array}
\end{equation}
where $\displaystyle{\tilde{m}= c(1-m)+cm\int_0^{\infty}s\sigma(s) ds>0}$ and by using \eqref{CONDITION-H} $\tilde{m}<\infty$ . It is easy to see that $w\in H_0^1(0,L)$. It follows, from the previous result and equations \eqref{max7}, \eqref{1max8} and \eqref{max2}, that 
\begin{equation}\label{5max9}
\eta,\ \eta_s\in W\quad \text{and}\quad \Lambda_{xx}\in L^2(0,L).
\end{equation}
%Now, Let 
%From \eqref{max2} and \eqref{max4}, we have 
%\begin{eqnarray}
%u-\alpha u_{xx}+\gamma \beta y_{xx}+\delta \omega_x &=&f^2+f^1,\label{max8}\\
%y-\beta y_{xx}+\gamma \beta u_{xx}&=&f^4+f^3,\label{max9}\\	
%w-c(1-m)w_{xx}-cm\int_0^{\infty}\sigma(s)\eta_{xx}(s)ds+\delta u_x&=&f^5+\delta f^1_x,\label{max10}\\
%\eta+\eta_s-w &=&f^6.\label{max11}	
%\end{eqnarray}
Now, let $\phi^1, \phi^2\in H_L^1(0,L)$  for all $i=1,2$. Multiplying \eqref{max3} and \eqref{max5} respectively  by $\phi^1$ and $\phi^2$, integrating by parts  over $(0,L)$, we get 
\begin{eqnarray}
\alpha \int_0^Lu_x\overline{\phi^1_x}dx-\gamma\beta\int_0^Ly_x\overline{\phi^1_x}dx&=&\int_0^L(\rho f^2+f^1)\overline{\phi^1}dx-\delta \int_0^Lw_x\overline{\phi^1}dx,\label{max12}\\
\alpha \int_0^Ly_x\overline{\phi^2_x}dx-\gamma\beta\int_0^Lu_x\overline{\phi^2_x}dx&=&\int_0^L(\mu f^4+f^3)\overline{\phi^2}dx\label{max13}.
\end{eqnarray}
%From \eqref{max7}, we have 
%\begin{equation}\label{max15}
%\eta(s)=\left(1-e^{-s}\right)w+\int_0^se^{-(s-\tau)}f^6(\cdot,s)d\tau.	
%\end{equation}
%Inserting \eqref{max15} in \eqref{max14}, we get 
%\begin{equation}\label{max16}
%\begin{array}{l}
%\displaystyle
%\int_0^Lw\overline{\phi^3}dx+c(1-m)\int_0^Lw_x\overline{\phi^3_x}dx+cm\int_0^L\int_0^{\infty}(1-e^{-s})\sigma(s)w_x\overline{\phi^3_x}dx-\delta\int_0^Lu\overline{\phi^3_x}dx=\\[0.1in]
%\displaystyle
%\int_0^L(f^5+\delta f^1_x)\overline{\phi^3}dx-cm\int_0^L\int_0^{\infty}\int_0^se^{-(s-\tau)}\sigma(s)f_x^6d\tau ds dx.
%\end{array}
%\end{equation}
Adding \eqref{max12} and  \eqref{max13}, and using the fact that $\alpha=\alpha_1+\gamma^2\beta$, we obtain 
\begin{equation}\label{max17}
\mathcal{B}((u,y),(\phi^1,\phi^2))=\mathcal{L}(\phi_1,\phi_2),\quad  \forall\, (\phi^1,\phi^2)\in H_L^1(0,L)\times H_L^1(0,L),
\end{equation}
where 
\begin{equation*}
\mathcal{B}((u,y),(\phi^1,\phi^2))=\alpha_1\int_0^Lu_x\overline{\varphi^1_x}dx+\beta \left[\int_0^L\left(\gamma^2u_x\overline{\phi^1}-\gamma y_x\overline{\phi^1_x}-\gamma u_x\overline{\phi^2_x}+y_x\overline{\phi^2_x}\right)dx\right]
\end{equation*}
and 
\begin{equation*}
\mathcal{L}(\phi_1,\phi_2)=\int_0^L(f^2+f^1)\overline{\phi^1}dx-\delta \int_0^Lw_x\overline{\phi^1}dx+\int_0^L(f^4+f^3)\overline{\phi^2}dx.
\end{equation*}
It is easy to see that, $\mathcal{B}$ is a sesquilinear, continuous and coercive form on $\left(H_L^1(0,L)\times H_L^1(0,1)\right)^2$ and $\mathcal{L}$ is a antilinear and continuous form on $H_L^1(0,L)\times H_L^1(0,1)$. Then, it follows by Lax-Milgram theorem that \eqref{max17} admits a unique solution $(u,y)\in \left(H_L^1(0,L)\times H_L^1(0,1)\right)$. From \eqref{max3}, \eqref{max5}, \eqref{4max9} and the fact that, $\alpha=\alpha_1+\gamma^2\beta$, we have 
\begin{equation*}
-\alpha_1u_{xx}=f^2+\gamma f^4-\frac{\delta}{\tilde{m}}\left[\Lambda^m(x)-m\int_0^\infty\sigma(s)\int_0^sf^6(\tau)d\tau ds\right]_x\in L^2(0,L)	
\end{equation*}
and 
\begin{equation*}
-\beta y_{xx}=\frac{\alpha}{\alpha_1}f^4+\frac{\gamma\beta}{\alpha_1}\left(f^2-\frac{\delta}{\tilde{m}}\left[\Lambda^m(x)-m\int_0^\infty\sigma(s)\int_0^sf^6(\tau)d\tau ds\right]_x\right)\in L^2(0,L). 
\end{equation*}
They follows that $u,y\in H^2(0,L)$.  Consequently,  $U=(u,-f^1,y,-f^3,w,\eta)^{\top} \in D({\mathcal{A}}_m) $ is a unique solution of \eqref{max1}. Then, ${\mathcal{A}}_m$ is an isomorphism and since $\rho\left({\mathcal{A}}_m\right)$ is open set of $\mathbb{C}$ (see Theorem 6.7 (Chapter III) in \cite{Kato01}),  we easily get $R(\lambda I -{\mathcal{A}}_m) = {\mathcal{H}}$ for a sufficiently small $\lambda>0 $. This, together with the dissipativeness of ${\mathcal{A}}_m$, imply that   $D\left({\mathcal{A}}_m\right)$ is dense in ${\mathcal{H}}$   and that ${\mathcal{A}}_m$ is m-dissipative in ${\mathcal{H}}$ (see Theorems 4.5, 4.6 in  \cite{Pazy01}). The proof is thus complete.

\noindent According to Lumer-Philips theorem (see \cite{Pazy01}), Proposition \ref{m-dissipative} implies that the operator $\AA$ generates a $C_{0}$-semigroup of contractions $e^{t\AA}$ in $\HH$ which gives the well-posedness of \eqref{evolution}. Then, we have the following result:
\begin{theoreme}{\rm
For all $U_0 \in \HH$,  system \eqref{Pizo1} admits a unique weak solution $$U(t)=e^{t\AA}U_0 \in C^0 (\R^+ ,\HH).
	$$ Moreover, if $U_0 \in D(\AA)$, then the system \eqref{Pizo1} admits a unique strong solution $$U(t)=e^{t\AA}U_0 \in C^0 (\R^+ ,D(\AA))\cap C^1 (\R^+ ,\HH).$$}
\end{theoreme}

\end{proof}

\section{Exponential Stability of Piezoelectric with Coleman-Gurtin thermal law \eqref{Pizo}}\label{EXP-STA}
\noindent In this section,    we shall analyze the exponential stability of system \eqref{Pizo}. The main result of this section is the following theorem 
\begin{theoreme}\label{EXPONENTIAL-THM}
Assume that the conditions \eqref{CONDITION-H} holds and $m\in (0,1)$. Then the $C_0$-semigroup of contractions $(e^{t\mathcal{A)}})_{t\geq 0}$ is exponentially stable; i.e. there exists constants $M\geq 1$ and $\epsilon>0$ independent of $U_0$ such that
\begin{equation}\label{EXPO}
\|e^{t{\mathcal{A}}_m}U_0\|_{\mathcal{H}}\leq Me^{-\epsilon t}\|U_0\|_{\mathcal{H}}.	\end{equation}
\xqed{$\square$}
\end{theoreme}
\noindent According to Huang \cite{Huang01} and Pr\"{u}ss \cite{pruss01},  a $C_0-$semigroup of contractions $\left(e^{t{\mathcal{A}}_m}\right)_{t\geq 0}$ on $\mathcal{H}$ satisfy \eqref{EXPO} if 
\begin{equation}\label{EXPO1}\tag{${\rm E1}$}
i\mathbb{R}\subset \rho({\mathcal{A}}_m)	
\end{equation}
and
\begin{equation}\label{EXPO2}\tag{${\rm E2}$}
\sup_{\la \in \mathbb{R}}\|(i\lambda I-{\mathcal{A}}_m)^{-1}\|_{\mathcal{L}(\mathcal{H})}=O(1)
\end{equation}
hold. Let $\left(\la,U:=(u,v,y,z,w,\eta)\right)\in \mathbb{R}^{\ast}\times D({\mathcal{A}}_m)$, such that 
\begin{equation}\label{EXPO-1}
\left(i\la I-{\mathcal{A}}_m\right)U=F:= (f_1,f_2,f_3,f_4,f_5,f_6)\in \mathcal{H}. 
\end{equation}
That is 
\begin{eqnarray}
i\lambda u-v&=&f^{1}\quad \text{in}\ \ H_L^1(0,L),\label{NU-eq1}\\
i\lambda\rho v-\alpha u_{xx}+\gamma\beta y_{xx}+\delta \omega_x&=&\rho f^{2}\quad \text{in}\ \ L^2(0,L),\label{NU-eq2}\\
i\lambda y-z&=&f^{3}\quad \text{in}\ \ H_L^1(0,L),\label{NU-eq3}\\ 
i\lambda\mu z-\beta y_{xx}+\gamma\beta u_{xx}&=&\mu f^{4}\quad \text{in}\ \ L^2(0,L),\label{NU-eq4}\\
i\lambda w-c\Lambda^m_{xx}+\delta v_x&=&f^{5}\quad \text{in}\ \ L^2(0,L),\label{NU-eq5}\\
i\lambda \eta+\eta_s-w&=&f^{6}(\cdot,s)\quad \text{in}\ \ W.\label{NU-eq6}
\end{eqnarray}
Here and below, we occasionally write $p\lesssim q$ to indicate that $p\leq C  q$ for some (implicit) constant $C>0$. The next Lemmas are a technical results to be used in the proof of Theorem \ref{EXPONENTIAL-THM}. 
%%%%%%%%%%%%%%%%%%%%
%%%%%%%%%%%%%%%%%%%%
%%%%%%    Lemma 1
%%%%%%%%%%%%%%%%%%%%
%%%%%%%%%%%%%%%%%%%%
\begin{lemma}\label{LEMMA1-EST}
Assume that the conditions \eqref{CONDITION-H} holds and $m\in (0,1)$. The solution $(u,v,y,z,w,\eta)\in D({\mathcal{A}}_m)$ of equation \eqref{EXPO-1} satisfies the following estimates 
\begin{equation}\label{Lemm1-EST1}
\int_0^L\abs{\omega_x}^2dx\leq K_1\|F\|_{\mathcal{H}}\|U\|_{\mathcal{H}},
\end{equation}
\begin{equation}\label{Lemm1-EST2}
\int_0^{L}\int_0^\infty\sigma(s)\abs{\eta_x}^2ds dx\leq K_2\|F\|_{\mathcal{H}}\|U\|_{\mathcal{H}},
\end{equation}
\begin{equation}\label{Lemm1-EST3}
\int_0^L\abs{\omega}^2dx\leq K_3\|F\|_{\mathcal{H}}\|U\|_{\mathcal{H}},
\end{equation}
\begin{equation}\label{Lemm1-EST4}
\int_0^L\abs{\Lambda^m_x}^2dx\leq K_4\|F\|_{\mathcal{H}}\|U\|_{\mathcal{H}},
\end{equation}
\end{lemma}
\noindent where 
$$
K_1= \frac{1}{c(1-m)}, \quad K_2=\frac{2}{cmd_{\sigma}},\quad K_3= c_p K_1\quad\text{and}\ \  K_4=\frac{2(1-m)}{c}+\frac{4g(0)}{cmd_{\sigma}}.
$$
\begin{proof}
First, taking the inner product of \eqref{EXPO-1} with $U$ in $\mathcal{H}$, we get
\begin{equation}\label{000Lemma1-eq4}
c(1-m)\int_0^L\abs{\omega_x}^2dx-\frac{cm}{2}\int_0^L\int_0^{\infty}	\sigma'(s)\abs{\eta_x}^2dsdx=-\Re\left(\left<{\mathcal{A}}_mU^n,U^n\right>_{\mathcal{H}}\right)\leq \|F_n\|_{\mathcal{H}}\|U\|_{\mathcal{H}}.
\end{equation}
From condition \eqref{CONDITION-H}, we obtain 
\begin{equation*}
\int_0^L\int_0^{\infty}\sigma(s)\abs{\eta_x}^2dsdx\leq -\frac{1}{d_{\sigma}}\int_0^L\int_0^{\infty}	\sigma'(s)\abs{\eta_x}^2dsdx.	
\end{equation*}
Using the above estimation in \eqref{000Lemma1-eq4}, we get 
\begin{equation}\label{Lemma1-eq4}
c(1-m)\int_0^L\abs{\omega_x}^2dx+\frac{cmd_{\sigma}}{2}\int_0^L\int_0^{\infty}\sigma(s)\abs{\eta_x}^2dsdx=-\Re\left(\left<{\mathcal{A}}_mU,U\right>_{\mathcal{H}}\right)\leq \|F\|_{\mathcal{H}}\|U\|_{\mathcal{H}},
\end{equation}
using the fact that $m\in(0,1)$ in \eqref{Lemma1-eq4}, then we get \eqref{Lemm1-EST1} and \eqref{Lemm1-EST2}.  Using \eqref{Lemm1-EST1} and Poincar\'e  inequality, we obtain \eqref{Lemm1-EST3}. Finally, by using Cauchy-Schwarz inequality, we obtain 
\begin{equation*}
\begin{array}{lll}
\displaystyle
\int_0^L\abs{\Lambda^m_x}^2dx&\leq&\displaystyle 2(1-m)^2\int_0^L\abs{w_x}^2dx+2\left(\int_0^{\infty}\sigma(s)ds\right)\int_0^L\int_0^{\infty}\sigma(s)\abs{\eta_x(s)}^2dsdx\\
&\leq&\displaystyle  2(1-m)^2\int_0^L\abs{w_x}^2dx+2g(0)\int_0^L\int_0^{\infty}\sigma(s)\abs{\eta_x(s)}^2dsdx.
\end{array}
\end{equation*}
Using \eqref{Lemm1-EST1} and \eqref{Lemm1-EST2} in the above inequality, we get \eqref{Lemm1-EST4}, the proof is thus completed. 
\end{proof}
%%%%%%%%%%%%%%%%%%%%
%%%%%%%%%%%%%%%%%%%%
%%%%%%    Lemma 2
%%%%%%%%%%%%%%%%%%%%
%%%%%%%%%%%%%%%%%%%%
\begin{lemma}\label{LEMMA2-EST}
Assume that the conditions \eqref{CONDITION-H} holds and $m\in (0,1)$. The solution $(u,v,y,z,w,\eta)\in D({\mathcal{A}}_m)$ of equation \eqref{EXPO-1} satisfies the following estimation 
\begin{equation}\label{Lemm2-1EST1}
\begin{array}{l}
\displaystyle 
\int_0^L\abs{u_x}^2dx\lesssim \left(\abs{\la}^{-1}+1\right)\|F\|_{\mathcal{H}}\|U\|_{\mathcal{H}}+\abs{\la}^{-1}\|F\|_{\mathcal{H}}^{\frac{1}{2}}\|U\|_{\mathcal{H}}^{\frac{1}{2}}\left((\abs{\la}+1)\|U\|_{\mathcal{H}}+\|F\|_{\mathcal{H}}\right)\\
\displaystyle 
+\abs{\la}^{-1}\|F\|^{\frac{1}{4}}_{\mathcal{H}}\left(A_{\lambda}^2(U,F)\|U\|_{\mathcal{H}}^{\frac{3}{4}}+A_{\lambda}(U,F)\|U\|_{\mathcal{H}}^{\frac{5}{4}}\right),
\end{array}
\end{equation}
where $A_{\lambda}(U,F)=\left(\abs{\la}^{\frac{1}{2}}+1\right)\|U\|^{\frac{1}{2}}_{\mathcal{H}}+\|F\|_{\mathcal{H}}^{\frac{1}{2}}$.
\end{lemma}
\begin{proof}
From \eqref{NU-eq1} and \eqref{NU-eq5}, we obtain 
\begin{equation}\label{LEMMA2-EQ1}
i\lambda \delta u_{x}=-i\la w+c\Lambda^m_{xx}+f^5+\delta f^1_x.
\end{equation}
Multiplying \eqref{LEMMA2-EQ1} by $-i\la^{-1}\overline{u_x}$, integrating by parts over $(0,L)$, we get 
\begin{equation*}
\delta \int_0^L\abs{u_x}^2dx=-\int_0^Lw\overline{u_x}dx+i\lambda^{-1}c\int_0^L\Lambda^m_x\overline{u_{xx}}dx+i\lambda^{-1}c\Lambda^m_x(0) \overline{u_x}(0)-i\lambda^{-1}\int_0^L\left(f^5+\delta f^1_x\right)\overline{u_x}dx.
\end{equation*}
It follows that, 
\begin{equation}\label{LEMMA2-EQ2}
\delta \int_0^L\abs{u_x}^2dx \leq \int_0^L\abs{w}\abs{u_x}dx+\abs{\la}^{-1}c\int_0^L\abs{\Lambda^m_x}\abs{u_{xx}}dx+\abs{\la}^{-1}c\abs{\Lambda^m_x(0)}\abs{u_x(0)}+\abs{\la}^{-1}\int_0^L\abs{f^5+\delta f_x^1}\abs{u_x}dx.
\end{equation}
Using  Cauchy-Schwarz inequality,  and the fact that $\|u_x\|_{L^2(0,L)}\leq \frac{1}{\sqrt{\alpha_1}}\|U\|_{\mathcal{H}}$, $\|f^1_x\|\leq \frac{1}{\sqrt{\alpha_1}}\|F\|_{\mathcal{H}}$, $\|f_5\|\leq \|F\|_{\mathcal{H}}$ and \eqref{Lemm1-EST3}, we get the following estimations 
\begin{equation}\label{LEMMA2-EQ3}
\left\{\begin{array}{l}
\displaystyle 
\int_0^L\abs{w}\abs{u_x}dx\leq \displaystyle 
\frac{1}{2\delta}\int_0^L\abs{\omega}^2dx+\frac{\delta}{2}\int_0^L\abs{u_x}^2dx\leq \frac{K_3}{2\delta}\|F\|_{\mathcal{H}}\|U\|_{\mathcal{H}}+\frac{\delta}{2}\int_0^L\abs{u_x}^2dx,\\[0.1in]
\displaystyle 
|\la|^{-1}\int_0^L\abs{f^5}\abs{u_x}dx\leq K_5\abs{\la}^{-1}\|F\|_{\mathcal{H}}\|U\|_{\mathcal{H}}\quad \text{where}\ K_5=\frac{1}{\sqrt{\alpha_1}},\\[0.1in]
\displaystyle 
|\la|^{-1}\delta\int_0^L\abs{f^1_x}\abs{u_x}dx\leq K_6\abs{\la}^{-1}\|F\|_{\mathcal{H}}\|U\|_{\mathcal{H}}\quad \text{where}\ K_6=\frac{\delta}{\alpha_1}.
\end{array}
\right.	
\end{equation}
Now, using the fact that $\alpha=\alpha_1+\gamma^2\beta$, \eqref{NU-eq2} and \eqref{NU-eq4}, we get 
\begin{equation}\label{LEMMA2-EQ4}
u_{xx}=\frac{1}{\alpha_1}\left[i\la \left(\rho v+\gamma\mu z\right)+\delta w_x-\rho f^2-\gamma \mu f^4\right].
\end{equation}
Using the fact that $\displaystyle{\rho\int_0^L\abs{v}^2dx\leq \|U\|^2_{\mathcal{H}},\ \mu\int_0^L\abs{z}^2\leq \|U\|^2_{\mathcal{H}},\rho\int_0^L\abs{f^2}^2dx\leq \|F\|^2_{\mathcal{H}},\mu\int_0^L\abs{f^4}^2dx\leq \|F\|^2_{\mathcal{H}}}$ and $ab\leq a^2+\frac{b^2}{4}$, in \eqref{LEMMA2-EQ4}, we get 
\begin{equation*}
\begin{array}{lll}
\|u_{xx}\|&\leq &\displaystyle  \frac{1}{\alpha_1}\left(\abs{\la}(\sqrt{\rho}+\gamma\sqrt{\mu})\|U\|_{\mathcal{H}}+\delta\sqrt{K_1}\|F\|^{\frac{1}{2}}_{\mathcal{H}}\|U\|^{\frac{1}{2}}_{\mathcal{H}}+(\sqrt{\rho}+\gamma\sqrt{\mu})\|F\|_{\mathcal{H}}\right)\\[0.1in]
&\leq &\displaystyle\frac{1}{\alpha_1} \max\left(\sqrt{\rho}+\gamma\sqrt{\mu},\delta \sqrt{K_1}\right)\left(\abs{\la}\|U\|_{\mathcal{H}}+\|F\|^{\frac{1}{2}}_{\mathcal{H}}\|U\|^{\frac{1}{2}}_{\mathcal{H}}+\|F\|_{\mathcal{H}}\right)	\\[0.1in]
&\leq &\displaystyle\frac{1}{\alpha_1} \max\left(\sqrt{\rho}+\gamma\sqrt{\mu},\delta \sqrt{K_1}\right)\left((\abs{\la}+1)\|U\|_{\mathcal{H}}+\frac{5}{4}\|F\|_{\mathcal{H}}\right).
\end{array}
\end{equation*}
Hence, we get 
\begin{equation}\label{LEMMA2-EQ5}
\|u_{xx}\|_{\mathcal{H}}\leq K_7\left((\abs{\la}+1)\|U\|_{\mathcal{H}}+\|F\|_{\mathcal{H}}\right),	
\end{equation}
where $K_7=\frac{5}{4\alpha_1} \max\left(\sqrt{\rho}+\gamma\sqrt{\mu},\delta \sqrt{K_1}\right)$. Using \eqref{Lemm1-EST4} and \eqref{LEMMA2-EQ5}, we obtain 
\begin{equation}\label{LEMMA2-EQ6-N}
\abs{\la}^{-1}c\int_0^L\abs{\Lambda^m_x}\abs{u_{xx}}dx\leq K_8 \abs{\la}^{-1}\|F\|_{\mathcal{H}}^{\frac{1}{2}}\|U\|_{\mathcal{H}}^{\frac{1}{2}}\left((\abs{\la}+1)\|U\|_{\mathcal{H}}+\|F\|_{\mathcal{H}}\right),
\end{equation}
where $K_8=cK_7\sqrt{K_4}$. From \eqref{LEMMA2-EQ1} and the fact that $\|f^5\|\leq \|F\|_{\mathcal{H}}$, $\|f_x^1\|\leq \alpha_1^{-\frac{1}{2}}\|F\|_{\mathcal{H}}$ and $\|w\|\leq \|U\|_{\mathcal{H}}$, we get 
\begin{equation}\label{LEMMA2-EQ6}
\|\Lambda^m_{xx}\|\leq K_9\left(\abs{\la}\|U\|_{\mathcal{H}}+\|F\|_{\mathcal{H}}\right),
\end{equation}
where $K_9=\left(\frac{\delta}{\sqrt{\alpha_1}}+1\right)$. Using Gagliardo-Nirenberg inequality, we get 
\begin{equation}\label{LEMMA2-EQ7}
\abs{\Lambda^m_x(0)}\leq C_1\|\Lambda^m_{xx}\|^{\frac{1}{2}}\|\Lambda^m_x\|^{\frac{1}{2}}+C_2\|\Lambda^m_x\|\leq \max(C_1,C_2)\|\Lambda^m_x\|^{\frac{1}{2}}\left(\|\Lambda^m_{xx}\|^{\frac{1}{2}}+\|\Lambda^m_x\|^{\frac{1}{2}}\right).
\end{equation}
Thanks to \eqref{LEMMA2-EQ6}, \eqref{Lemm1-EST4} and \eqref{LEMMA2-EQ7}, and the fact that $\sqrt{a+b}\leq \sqrt{a}+\sqrt{b}$ and $ab\leq a^2+\frac{b^2}{4}$, we get 
\begin{equation}\label{LEMMA2-EQ8}
\abs{\Lambda^m_x(0)}\leq K_{10}\|F\|_{\mathcal{H}}^{\frac{1}{4}}\|U\|_{\mathcal{H}}^{\frac{1}{4}}\left(\sqrt{\abs{\la}\|U\|_{\mathcal{H}}+\|F\|_{\mathcal{H}}}+\|F\|_{\mathcal{H}}^{\frac{1}{4}}\|U\|_{\mathcal{H}}^{\frac{1}{4}}\right)\leq K_{10}\|F\|_{\mathcal{H}}^{\frac{1}{4}}\|U\|_{\mathcal{H}}^{\frac{1}{4}}A_{\lambda}(U,F),
\end{equation}
where $K_{10}=\frac{5}{4}\sqrt{\sqrt{K_4}}\max(C_1,C_2)\max\left(\sqrt{K_9},\sqrt{\sqrt{K_4}}\right)$ and $A_{\lambda}(U,F)=\left(\abs{\la}^{\frac{1}{2}}+1\right)\|U\|^{\frac{1}{2}}_{\mathcal{H}}+\|F\|_{\mathcal{H}}^{\frac{1}{2}}$. Again, using Gagliardo-Nirenberg  inequality, we have  
$$
\abs{u_x(0)}\leq C_1\|u_{xx}\|^{\frac{1}{2}}\|u_x\|^{\frac{1}{2}}+C_2\|u_x\|.
$$
Using \eqref{LEMMA2-EQ5} and the fact that $\|u_x\|\leq \frac{1}{\sqrt{\alpha_1}}\|U\|_{\mathcal{H}}$, in the above inequality, we get 
\begin{equation}\label{LEMMA2-EQ9}
\abs{u_x(0)}\leq K_{11}\left(\sqrt{(\abs{\la}+1)\|U\|_{\mathcal{H}}+\|F\|_{\mathcal{H}}}\|U\|^{\frac{1}{2}}_{\mathcal{H}}+\|U\|_{\mathcal{H}}\right)\leq K_{11}\left(A_{\lambda}(U,F)\|U\|^{\frac{1}{2}}_{\mathcal{H}}+\|U\|_{\mathcal{H}}\right),	
\end{equation}
where $K_{11}=\displaystyle{\max\left(\frac{C_1\sqrt{K_7}}{\alpha_1^{\frac{1}{4}}},\frac{C_2}{\sqrt{\alpha_1}}\right)}$. Using \eqref{LEMMA2-EQ7} and \eqref{LEMMA2-EQ9}, we get 
\begin{equation}\label{LEMMA2-EQ10}
\abs{\la}^{-1}c\abs{\Lambda^m_x(0)}\abs{u_x(0)}\leq K_{12}\abs{\la}^{-1}\|F\|_{\mathcal{H}}^{\frac{1}{4}}\left(A_{\lambda}^2(U,F)\|U\|_{\mathcal{H}}^{\frac{3}{4}}+A_{\lambda}(U,F)\|U\|_{\mathcal{H}}^{\frac{5}{4}}\right),	
\end{equation}
where $K_{12}=cK_{10}K_{11}$. Finally, inserting  \eqref{LEMMA2-EQ3}, \eqref{LEMMA2-EQ6-N} and \eqref{LEMMA2-EQ10} in \eqref{LEMMA2-EQ2}, we get 
\begin{equation*}
\begin{array}{l}
\displaystyle
\int_0^L\abs{u_x}^2dx\leq K_{13}\left(\left(\abs{\la}^{-1}+1\right)\|F\|_{\mathcal{H}}\|U\|_{\mathcal{H}}+\abs{\la}^{-1}\|F\|_{\mathcal{H}}^{\frac{1}{2}}\|U\|_{\mathcal{H}}^{\frac{1}{2}}\left((\abs{\la}+1)\|U\|_{\mathcal{H}}+\|F\|_{\mathcal{H}}\right)\right.\\[0.1in]	
\left.+\abs{\la}^{-1}\|F\|^{\frac{1}{4}}_{\mathcal{H}}\left(A_{\lambda}^2(U,F)\|U\|_{\mathcal{H}}^{\frac{3}{4}}+A_{\lambda}(U,F)\|U\|_{\mathcal{H}}^{\frac{5}{4}}\right)\right),
\end{array}
\end{equation*}
where $K_{13}=\displaystyle{\frac{2}{\delta}\max\left(\max\left(K_5+K_6,\frac{K_3}{2\delta}\right),K_8,K_{12}\right)}$. Hence, we obtain \eqref{Lemm2-1EST1}. The proof has been completed.
\end{proof}
%%%%%%%%%%%%%%%%%%%%%%%%%%%%%%%%%%
%%%% Lemma 2 Est 2
%%%%%%%%%%%%%%%%%%%%%%%%%%%%%%%%%% 
$\newline$
\\
Inserting \eqref{NU-eq1} in \eqref{NU-eq2}, we get 
\begin{equation}\label{Lemm2-EST2}
\rho \la^2u+\alpha u_{xx}-\gamma \beta y_{xx}-\delta w_x=-\rho f^2-i\la f^1. 
\end{equation}
\begin{lemma}\label{LEMMA3-EST}
The solution $(u,v,y,z,w,\eta)\in D({\mathcal{A}}_m)$ of equation \eqref{EXPO-1} satisfies the following estimation 
\begin{equation}\label{Lemm2-EST1}
\int_0^L\abs{\la u}^2dx\lesssim \int_0^L\abs{u_x}^2dx+\|U\|_{\mathcal{H}}\|u_x\|+\left(|\la|^{-1}+1\right)\left(\|F\|_{\mathcal{H}}\|U\|_{\mathcal{H}}+\|F\|_{H}^2\right).	
\end{equation}
\end{lemma}
\begin{proof}
Multiplying \eqref{Lemm2-EST2} by $\overline{u}$, integrating by parts over $(0,L)$, we get 
\begin{equation*}
\rho \int_0^L\abs{\la u}^2dx=\alpha \int_0^L\abs{u_x}^2dx-\gamma \beta \int_0^Ly_x\overline{u_x}dx-\delta\int_0^Lw\overline{u_x}dx-\rho\int_0^Lf^2\overline{u}dx-i\la \int_0^Lf^1\overline{u}dx.	
\end{equation*}
It follows that, 
\begin{equation}\label{Lemm2-EST3}
\rho \int_0^L\abs{\la u}^2dx\leq \alpha \int_0^L\abs{u_x}^2dx+\gamma \beta \int_0^L\abs{y_x}\abs{u_x}dx+\delta \int_0^L\abs{w}\abs{u_x}dx+\rho\int_0^L\abs{f^2}\abs{u}dx+\abs{\la} \int_0^L\abs{f^1}\abs{u}dx.	
\end{equation}
Using the fact that $\|y_x-\gamma u_x\|\leq \frac{1}{\sqrt{\beta}}\|U\|_{\mathcal{H}}$, we get 
\begin{equation}\label{Lemm2-EST4}
\|y_x\|\leq \|y_x-\gamma u_x\|+\gamma\|u_x\|\leq \left(\frac{1}{\sqrt{\beta}}+\frac{\gamma}{\sqrt{\alpha_1}}\right)\|U\|_{\mathcal{H}}.
\end{equation}
Using Cauchy-Schwarz inequality and \eqref{Lemm2-EST4}, we get 
\begin{equation}\label{Lemm2-EST5}
\gamma \beta \int_0^L\abs{y_x}\abs{u_x}dx\leq \gamma\beta\left(\frac{1}{\sqrt{\beta}}+\frac{\gamma}{\sqrt{\alpha_1}}\right)\|U\|_{\mathcal{H}}\|u_x\|.	
\end{equation}
Using \eqref{Lemm1-EST3} and Youngs's inequality, we get 
\begin{equation}\label{Lemm2-EST6}
\delta \int_0^L\abs{w}\abs{u_x}dx\leq \int_0^L\abs{w}^2dx+\frac{\delta^2}{4}\int_0^L\abs{u_x}^2dx\leq K_3\|F\|_{\mathcal{H}}\|U\|_{\mathcal{H}}+\frac{\delta^2}{4}\int_0^L\abs{u_x}^2dx.
\end{equation}
From \eqref{NU-eq1}, the fact that $\sqrt{\rho}\|v\|\leq \|U\|_{\mathcal{H}}$, and Poincar\'e inequality, we have 
\begin{equation}\label{Lemm2-EST7}
\|\la u\|\leq \frac{1}{\sqrt{\rho}}\|U\|_{\mathcal{H}}+\frac{C_p}{\sqrt{\alpha_1}}\|f^1_x\|_{\mathcal{H}}\leq K_{14}\left(\|U\|_{\mathcal{H}}+\|F\|_{\mathcal{H}}\right),	
\end{equation}
where $K_{14}=\max\left(\frac{1}{\sqrt{\rho}},\frac{C_p}{\sqrt{\alpha_1}}\right)$. Using \eqref{Lemm2-EST7} and the fact that $\displaystyle{\rho \int_0^L\abs{f_2}^2dx\leq \|F\|^2_{\mathcal{H}}}$, we get 
\begin{equation}\label{Lemm2-EST8}
\rho\int_0^L\abs{f^2}\abs{u}dx\leq \sqrt{\rho}\|F\|_{H}\|u\|\leq K_{15}\la^{-1}\left(\|F\|_{\mathcal{H}}\|U\|_{\mathcal{H}}+\|F\|^2_{\mathcal{H}}\right),	
\end{equation}
where $K_{15}=\sqrt{\rho}K_{14}$. Using  \eqref{Lemm2-EST7}, Poincar\'e inequality and the fact that $\displaystyle{\alpha_1\int_0^L\abs{f^1_x}^2dx\leq \|F\|^2_{\mathcal{H}}}$, we get 
\begin{equation}\label{Lemm2-EST9}
\rho \abs{\la}\int_0^L\abs{f^1}\abs{u}dx\leq \rho C_p\|\la u\|\|f_x^1\|\leq K_{16}\left(\|F\|_{\mathcal{H}}\|U\|_{\mathcal{H}}+\|F\|_{H}^2\right),	
\end{equation}
where $K_{16}=\frac{\rho C_pK_{14}}{\sqrt{\alpha_1}}$. 
Adding \eqref{Lemm2-EST8} and \eqref{Lemm2-EST9}, we get 
\begin{equation}\label{Lemm2-EST10}
\rho\int_0^L\abs{f^2}\abs{u}dx+\rho \abs{\la}\int_0^L\abs{f^1}\abs{u}dx\leq K_{17}\left(|\la|^{-1}+1\right)\left(\|F\|_{\mathcal{H}}\|U\|_{\mathcal{H}}+\|F\|_{H}^2\right),
\end{equation}
where $K_{17}=\max\left(K_{15},K_{16}\right)$. Finally, inserting \eqref{Lemm2-EST5}, \eqref{Lemm2-EST6} and \eqref{Lemm2-EST10}, in \eqref{Lemm2-EST3}, we get 
\begin{equation*}
\int_0^L\abs{\la u}^2dx\leq K_{18}\int_0^L\abs{u_x}^2dx+K_{19}\|U\|_{\mathcal{H}}\|u_x\|+K_{20}\left(|\la|^{-1}+1\right)\left(\|F\|_{\mathcal{H}}\|U\|_{\mathcal{H}}+\|F\|_{H}^2\right),	
\end{equation*}
where $K_{18}=\rho^{-1}\left(\alpha+\frac{\delta^2}{4}\right)$, $K_{19}=\rho^{-1} \gamma\beta\left(\frac{1}{\sqrt{\beta}}+\frac{\gamma}{\sqrt{\alpha_1}}\right)$ and $K_{20}=2\rho^{-1}\max(K_{17},K_3)$. It follow that, 
\begin{equation}\label{Lemm2-EST11}
\int_0^L\abs{\la u}^2dx\leq K_{21}\left(\int_0^L\abs{u_x}^2dx+\|U\|_{\mathcal{H}}\|u_x\|+\left(|\la|^{-1}+1\right)\left(\|F\|_{\mathcal{H}}\|U\|_{\mathcal{H}}+\|F\|_{H}^2\right)\right),	
\end{equation}
where $K_{21}=\max(K_{18},K_{19},K_{20})$. Hence, we obtain  \eqref{Lemm2-EST1}. The proof is thus completed. 
\end{proof}
%%%%%%%%%%%%%%%%%%%%%%%%%%%%%%%%%%%
%%%%%%%%% Lemma 3
%%%%%%%%%%%%%%%%%%%%%%%%%%%%%%%%%%%
\begin{lemma}\label{LEMMA4-EST}
The solution $(u,v,y,z,w,\eta)\in D({\mathcal{A}}_m)$ of equation \eqref{EXPO-1} satisfies the following estimation 
\begin{equation}\label{Lemm3-EST1}
\int_0^L\abs{y_x}^2dx\lesssim \int_0^L\abs{\la u}^2dx+\int_0^L\abs{u_x}^2dx+\left(\abs{\la}^{-1}+1\right)\left(\|F\|_{\mathcal{H}}\|U\|_{\mathcal{H}}+\|F\|^2_{\mathcal{H}}\right)	\end{equation}
and 
\begin{equation}\label{1Lemm3-EST1}
\int_0^L\abs{\la y}^2dx\lesssim \int_0^L\abs{\la u}^2dx+\int_0^L\abs{u_x}^2dx+\left(\abs{\la}^{-1}+1\right)\left(\|F\|_{\mathcal{H}}\|U\|_{\mathcal{H}}+\|F\|^2_{\mathcal{H}}\right).
\end{equation}

\end{lemma}
\begin{proof}
The proof is divided into three steps.\\
\textbf{Step 1.} The aim of this step is to prove the following estimation 
\begin{equation}\label{STEP1-EST1}
\begin{array}{l}
\displaystyle
\frac{\gamma \beta}{2}\int_0^L\abs{y_x}^2dx\leq \frac{\rho^2}{\mu\gamma}\int_0^L\abs{\la u}^2dx+\frac{\mu\gamma}{4}\int_0^L\abs{\la y}^2dx+\frac{\alpha^2}{\gamma\beta}\int_0^L\abs{u_x}^2dx+\\[0.1in]
\displaystyle 
 \widetilde{K}\left(\abs{\la}^{-1}+1\right)\left(\|F\|_{\mathcal{H}}\|U\|_{\mathcal{H}}+\|F\|^2_{\mathcal{H}}\right),
 \end{array}
\end{equation}
where $\widetilde{K}$ is a positive constant which defined at the end of the proof of step 1. For this aim, Multiplying \eqref{Lemm2-EST2} by $\overline{y}$, integrating over $(0,L)$, we get 
\begin{equation*}
\rho\la^2\int_0^Lu\overline{y}dx-\alpha\int_0^Lu_x\overline{y_x}dx+\gamma \beta\int_0^L\abs{y_x}^2dx+\delta \int_0^Lw\overline{y_x}dx=-\rho\int_0^Lf^2\overline{y}dx-i\la \int_0^Lf^1\overline{y}dx. 	
\end{equation*}
It follows that, 
\begin{equation}\label{Lemma3-EST1}
\gamma\beta \int _0^L\abs{y_x}^2dx\leq \rho \la^2\int_0^L\abs{u}\abs{y}dx+\alpha\int_0^L\abs{u_x}\abs{y_x}dx+\delta \int_0^L\abs{w}\abs{y_x}dx+\rho\int_0^L\abs{f^2}\abs{y}dx+\abs{\la}\int_0^L\abs{f^1}\abs{y}dx.	
\end{equation}
Applying Young's inequality, we get 
\begin{equation}\label{Lemma3-EST4}
\left\{\begin{array}{l}
\displaystyle \rho \la^2\int_0^L\abs{u}\abs{y}dx\leq \frac{\rho^2}{\mu\gamma}\int_0^L\abs{\la u}^2dx+ \frac{\mu\gamma}{4}\int_0^L\abs{\la y}^2dx,\\[0.1in]\displaystyle 
\alpha\int_0^L\abs{u_x}\abs{y_x}dx\leq \frac{\alpha^2}{\gamma\beta}\int_0^L\abs{u_x}^2dx+\frac{\gamma\beta}{4}\int_0^L\abs{y_x}^2dx.\\[0.1in]
\end{array}
\right.
\end{equation}
Using \eqref{Lemm1-EST3} and the fact that $ab\leq a^2+\frac{b^2}{4}$, we get 
\begin{equation}\label{Lemma3-EST5}
\delta \int_0^L\abs{w}\abs{y_x}dx\leq \frac{\delta^2}{\gamma\beta}\int_0^L\abs{\omega}^2dx+\frac{\gamma\beta}{4}\int_0^L\abs{y_x}^2dx\leq K_{22}\|F\|_{\mathcal{H}}\|U\|_{\mathcal{H}}+\frac{\gamma\beta}{4}\int_0^L\abs{y_x}^2dx,
\end{equation}
where $K_{22}=\frac{\delta^2}{\gamma\beta}K_3$. It is easy to see that
\begin{equation}\label{Lemma3-EST2}
\|f^3_x\|\leq \|f_x^3-\gamma f_x^1\|+\gamma \|f_x^1\|\leq K_{23}\|F\|,	
\end{equation}
where $K_{23}=\displaystyle{\frac{1}{\sqrt{\beta}}+\frac{\gamma}{\sqrt{\alpha_1}}}$.  Using \eqref{Lemma3-EST2}, Poincar\'e inequality and  the fact that $\displaystyle{\mu\int_0^L\abs{z}^2dx\leq \|U\|^2_{\mathcal{H}}}$ in \eqref{NU-eq3}, we obtain 
\begin{equation}\label{Lemma3-EST3}
\|\la y\|\leq \frac{1}{\sqrt{\mu}}\|U\|+C_p\|f^3_x\|\leq K_{24}	\left(\|U\|_{\mathcal{H}}+\|F\|_{\mathcal{H}}\right),
\end{equation}
where $K_{24}=\max\left(\frac{1}{\sqrt{\mu}},C_pK_{23}\right)$. Using \eqref{Lemma3-EST3}, Poincar\'e inequality  and the fact that $\displaystyle{\alpha_1\int_0^L\abs{f^1_x}^2dx\leq \|F\|_{\mathcal{H}}^2}$, we get 
\begin{equation}\label{Lemma3-EST6}
\abs{\la}\int_0^L\abs{f^1}\abs{y}dx\leq C_p\|f^1_x\|\|\la y\|\leq K_{25}\left(\|F\|_{\mathcal{H}}\|U\|_{\mathcal{H}}+\|F\|_{\mathcal{H}}^2\right),	
\end{equation}
where $K_{25}=\frac{C_pK_{24}}{\sqrt{\alpha_1}}$. On the other hand, using \eqref{Lemma3-EST3} and the fact that $\displaystyle{\rho\int_0^L\abs{f^2}^2dx\leq \|F\|_{\mathcal{H}}}$, we get 
\begin{equation}\label{Lemma3-EST7}
\rho\int_0^L\abs{f^2}\abs{y}dx\leq \abs{\la}^{-1}K_{26}\left(\|F\|_{\mathcal{H}}\|U\|_{\mathcal{H}}+\|F\|^2_{\mathcal{H}}\right),
\end{equation}
where $K_{26}=\sqrt{\rho}K_{24}$. Inserting \eqref{Lemma3-EST4}, \eqref{Lemma3-EST5}, \eqref{Lemma3-EST6} and \eqref{Lemma3-EST7} in \eqref{Lemma3-EST1}, we get \eqref{STEP1-EST1}, such that $\widetilde{K}=2\max\left(K_{22},K_{25},K_{26}\right)$.\\
\textbf{Step 2.} The aim of this step is to prove the following estimation 
\begin{equation}\label{STEP2-EST1}
\mu\int_0^L\abs{\la y}^2dx\leq \frac{5\beta}{4}\int_0^L\abs{y_x}^2dx+\gamma^2\beta\int_0^L\abs{u_x}^2dx+\widetilde{K_1}\left(\abs{\la}^{-1}+1\right)\left(\|F\|_{\mathcal{H}}\|U\|_{\mathcal{H}}+\|F\|^2_{\mathcal{H}}\right),	
\end{equation}
where $\widetilde{K}$ is a positive constant which is defined at the end of the proof of step 1. For this aim, inserting \eqref{NU-eq3} in \eqref{NU-eq4}, we get 
\begin{equation}\label{STEP2-EST2}
\mu\la^2 y+\beta y_{xx}-\gamma\beta u_{xx}=-\left(\mu f^4+i\la \mu f^3\right).	
\end{equation}
Multiplying \eqref{STEP2-EST2} by $\overline{y}$ and integrating by parts over $(0,L)$, we get 
\begin{equation*}
\mu\int_0^L\abs{\la y}^2dx=\beta \int_0^L\abs{y_x}^2dx-\gamma\beta\int_0^Lu_x\overline{y_x}dx-\mu\int_0^Lf^4\overline{y}dx-i\la \mu\int_0^Lf^3\overline{y}dx.	
\end{equation*}
It follows that, 
\begin{equation}\label{STEP2-EST3}
\mu\int_0^L\abs{\la y}^2dx\leq \beta \int_0^L\abs{y_x}^2dx+\gamma\beta\int_0^L\abs{u_x}\abs{y_x}dx+\mu\int_0^L\abs{f^4}\abs{y}dx+\abs{\la} \mu\int_0^L\abs{f^3}\abs{y}dx.
\end{equation}
Using the fact that $ab\leq a^2+\frac{b^2}{4}$, we get 
\begin{equation}\label{STEP2-EST4}
\gamma\beta\int_0^L\abs{u_x}\abs{y_x}dx\leq \gamma^2\beta\int_0^L\abs{u_x}^2+\frac{\beta}{4}\int_0^L\abs{y_x}^2dx.	
\end{equation}
Using \eqref{Lemma3-EST3} and the fact that $\displaystyle{\mu\int_0^L\abs{f_4}^2dx\leq \|F\|_{\mathcal{H}}^2}$, we get 
\begin{equation}\label{STEP2-EST5}
\mu\int_0^L\abs{f^4}\abs{y}dx\leq \sqrt{\mu}\|F\|_{\mathcal{H}}\|y\|\leq K_{27}\abs{\la}^{-1}\left(\|F\|_{\mathcal{H}}\|U\|_{\mathcal{H}}+\|F\|_{\mathcal{H}}^2\right),	
\end{equation}
where $K_{27}=\sqrt{\mu}K_{24}$. Using Poincar\'e inequality, \eqref{Lemma3-EST2} and \eqref{Lemma3-EST3}, we get 
\begin{equation}\label{STEP2-EST6}
\abs{\la} \mu\int_0^L\abs{f^3}\abs{y}dx\leq \mu C_p\|f_x^3\|\|\la y\|\leq K_{28}\left(\|F\|_{\mathcal{H}}\|U\|_{\mathcal{H}}+\|F\|_{\mathcal{H}}^2\right),
\end{equation}
where $K_{28}=\mu C_pK_{23}K_{24}$. Inserting \eqref{STEP2-EST4}, \eqref{STEP2-EST5} and \eqref{STEP2-EST6} in \eqref{STEP2-EST3}, we obtain \eqref{STEP2-EST1} with $\widetilde{K_1}=\max\left(K_{27},K_{28}\right)$.\\
\textbf{Step 3.} The aim of this step is to prove \eqref{Lemm3-EST1} and \eqref{1Lemm3-EST1}. Inserting \eqref{STEP2-EST1} in \eqref{STEP1-EST1}, we get 
\begin{equation*}
\frac{3}{16}\gamma\beta\int_0^L\abs{y_x}^2dx\leq \frac{\rho^2}{\mu\gamma}\int_0^L\abs{\la u}^2dx+\frac{\gamma^3\beta}{4}\int_0^L\abs{u_x}^2dx+\widetilde{K_2}\left(\abs{\la}^{-1}+1\right)\left(\|F\|_{\mathcal{H}}\|U\|_{\mathcal{H}}+\|F\|^2_{\mathcal{H}}\right),	
\end{equation*}
where $\widetilde{K_2}=\frac{\gamma}{4}\widetilde{K_1}+\widetilde{K}$. It follows that 
\begin{equation}\label{STEP3-EQ1}
\int_0^L\abs{y_x}^2dx\leq K_{29}\left(\int_0^L\abs{\la u}^2dx+\int_0^L\abs{u_x}^2dx+\left(\abs{\la}^{-1}+1\right)\left(\|F\|_{\mathcal{H}}\|U\|_{\mathcal{H}}+\|F\|^2_{\mathcal{H}}\right)\right), 
\end{equation}
where $\displaystyle{K_{29}=\frac{16}{3\gamma\beta}\max\left( \frac{\rho^2}{\mu\gamma},\frac{\gamma^3\beta}{4},\widetilde{K_2}\right)}$. Hence, we obtain \eqref{Lemm3-EST1}. Inserting \eqref{STEP3-EQ1} in \eqref{STEP2-EST1}, we get 
\begin{equation}
\int_0^L\abs{\la u}^2dx\leq K_{30}\left(\int_0^L\abs{\la u}^2dx+\int_0^L\abs{u_x}^2dx+\left(\abs{\la}^{-1}+1\right)\left(\|F\|_{\mathcal{H}}\|U\|_{\mathcal{H}}+\|F\|^2_{\mathcal{H}}\right)\right),
\end{equation}
where $K_{30}=\displaystyle{\frac{1}{\mu}\max\left(\frac{5\beta}{4}K_{29}+\gamma^2\beta,\frac{5\beta}{4}K_{29}+\widetilde{K_1}\right)}$. Hence, we obtain \eqref{1Lemm3-EST1}. The proof is thus completed. 
\end{proof}
\\

\noindent \textbf{Proof of Theorem \ref{EXPONENTIAL-THM}.} First, we will prove \eqref{EXPO1}. Remark that it has been proved in Proposition \eqref{m-dissipative} that $0\in\rho\left({\mathcal{A}}_m\right)$. Now, suppose \eqref{EXPO1} is not true, then there exists $\kappa\in \mathbb{R}^{\ast}$ such that $i\kappa\notin \rho\left({\mathcal{A}}_m\right)$. According to Remark \ref{App-Lemma-A.3} in Appendix A, there exists 
\begin{equation*}
\left\{\left(\la_n,U^n:=(u^n,v^n,y^n,z^n,w^n,\eta^n(s))\right)\right\}_{n\geq 1}\subset \mathbb{R}^{\ast}\times D({\mathcal{A}}_m),	
\end{equation*}
 with $\la_n\to \kappa$ as $n\to \infty$, $\abs{\la_n}<\abs{\kappa}$ ans $\|U^n\|_{\mathcal{H}}=1$, such that 
 \begin{equation*}
 \left(i\la_nI-{\mathcal{A}}_m\right)U^n=F_n:=\left(f^1_n,f^2_n,f^3_n,f^4_n,f^5_n,f^6_n(\cdot,s)\right)\to 0\quad \text{in}\quad \mathcal{H},\quad \text{as}\ n\to \infty.	
 \end{equation*}
We will check \eqref{EXPO1} by finding a contradiction with $\|U^n\|_{\mathcal{H}}=1$ such as $\|U^n\|_{\mathcal{H}}\to 0$. Here and below take $U=U^n$, $F=F_n$ and $\la=\la_n$.  According to Lemma \ref{LEMMA1-EST}, we get 
\begin{equation}\label{Proof1-EXPO}
\int_0^L\abs{w^n}^2dx\to 0\quad \text{and}\quad \int_0^{\infty}\int_0^L\sigma(s)\abs{\eta^n_x}^2dsdx\to 0. 
\end{equation}
According to Lemma \ref{LEMMA2-EST} and using the facts that $\abs{\la_n}<\abs{\kappa}$, $\|U^n\|_{\mathcal{H}}=1$ and $\|F_n\|_{\mathcal{H}}\to 0$, we have 
$$
A_{\lambda_n}(U^n,F_n)\to \abs{\kappa}^{\frac{1}{2}}+1,\quad \text{as}\ \ n\to \infty.
$$
hence 
\begin{equation}\label{Proof2-EXPO}
\int_0^L\abs{u^n_x}^2dx\to 0\quad \text{as}\ \ n\to \infty.	
\end{equation}
Using \eqref{Proof2-EXPO} and the facts that $\abs{\la_n}<\abs{\kappa}$, $\|U^n\|_{\mathcal{H}}=1$ and $\|F_n\|_{\mathcal{H}}\to 0$ in Lemma \ref{LEMMA3-EST}, we obtain 
\begin{equation}\label{Proof3-EXPO}
\int_0^L\abs{\la_nu^n}^2dx\to 0\quad \text{as}\ \ n\to \infty.  	
\end{equation}
Using \eqref{Proof2-EXPO}, \eqref{Proof3-EXPO} and the facts that $\abs{\la_n}<\abs{\kappa}$, $\|U^n\|_{\mathcal{H}}=1$ and $\|F_n\|_{\mathcal{H}}\to 0$ in Lemma \ref{LEMMA4-EST}, we get
\begin{equation}\label{Proof4-EXPO}
\int_0^L\abs{y_x^n}^2dx\to 0\quad \text{and}\quad \int_0^L\abs{\la_n y^n}^2dx\to 0\quad \text{as}\ \ n\to \infty.	
\end{equation}
From \eqref{Proof1-EXPO}-\eqref{Proof4-EXPO}, as $n\to \infty$, we get $\|U^n\|_{\mathcal{H}}\to 0$, which contradicts $\|U^n\|_{\mathcal{H}}=1$. Thus, condition \eqref{EXPO1} holds true. Next, we will prove \eqref{EXPO2} by a contradiction argument. Suppose there exists 
\begin{equation*}
\left\{\left(\la_n,U^n:=(u^n,v^n,y^n,z^n,w^n,\eta^n(s))\right)\right\}_{n\geq 1}\subset \mathbb{R}^{\ast}\times D({\mathcal{A}}_m),	
\end{equation*}
with $\abs{\la_n}\geq 1$ without affecting the result, such that $\abs{\la_n}\to \infty$, and $\|U^n\|_{\mathcal{H}}=1$ and there exists a sequence $F_n=\left(f^1_n,f^2_n,f^3_n,f^4_n,f^5_n,f^6_n(\cdot,s)\right)\in\mathcal{H}$, such that 
$$
\left(i\la_n I-{\mathcal{A}}_m\right)U^n=F_n\to 0\quad \text{in}\quad \mathcal{H}. 
$$
We use conventional asymptotic notation, including 'big O' and 'little o'. We will check \eqref{EXPO2} by finding a contradiction with $\|U^n\|_{\mathcal{H}}=1$ such as $\|U^n\|_{\mathcal{H}}=o(1)$.  According to Lemma \ref{LEMMA1-EST} and using the facts that $\abs{\la_n}\to \infty$, $\|U^n\|_{\mathcal{H}}=1$ and $\|F_n\|\to 0$, we have  
\begin{equation}\label{1Proof1-EXPO}
\int_0^L\abs{w^n}^2dx=o(1)\quad \text{and}\quad \int_0^{\infty}\int_0^L\sigma(s)\abs{\eta^n_x}^2dsdx=o(1). 
\end{equation}
According to  Lemma \ref{LEMMA2-EST} and using the facts that $\abs{\la_n}\to \infty$, $\|U^n\|_{\mathcal{H}}=1$ and $\|F_n\|_{\mathcal{H}}\to 0$, we have 
$$
A_{\lambda_n}(U^n,F_n)=\left(\abs{\la_n}^{\frac{1}{2}}+1\right)\|U^n\|^{\frac{1}{2}}_{\mathcal{H}}+\|F_n\|_{\mathcal{H}}^{\frac{1}{2}}\leq O(\abs{\la_n}^{\frac{1}{2}}).
$$
hence 
\begin{equation}\label{2Proof2-EXPO}
\int_0^L\abs{u^n_x}^2dx=o(1).	
\end{equation}
Using \eqref{2Proof2-EXPO} and the fact that $\abs{\la_n}\to \infty$, $\|U^n\|_{\mathcal{H}}=1$ and $\|F_n\|_{\mathcal{H}}\to 0$ in Lemma \ref{LEMMA3-EST}, we obtain 
\begin{equation}\label{3Proof3-EXPO}
\int_0^L\abs{\la_nu^n}^2dx=o(1).  	
\end{equation}
Using \eqref{2Proof2-EXPO}, \eqref{3Proof3-EXPO} and the facts that $\abs{\la_n}\to \infty$, $\|U^n\|_{\mathcal{H}}=1$ and $\|F_n\|_{\mathcal{H}}\to 0$ in Lemma \ref{LEMMA4-EST}, we get
\begin{equation}\label{4Proof4-EXPO}
\int_0^L\abs{y_x^n}^2dx=o(1)\quad \text{and}\quad \int_0^L\abs{\la_n y^n}^2dx=o(1).	
\end{equation}
From \eqref{1Proof1-EXPO}-\eqref{4Proof4-EXPO}, as $\abs{\la_n}\to \infty$, we get $\|U^n\|_{\mathcal{H}}=o(1)$, which contradicts $\|U^n\|_{\mathcal{H}}=1$. Thus, condition \eqref{EXPO2} holds true. The result follows from Theorem \ref{bt} (part (i)) in Appendix section. The proof is thus complete.
\xqed{$\square$}
%%%%%%%%%%%%%%%%%%%%%%%%%%%%%%%%%%%%%%%%%%%
\begin{rem}
In the limit case where $m\to 0$, the \eqref{Pizo} is transformed to Piezoelectric with Fourier Law \eqref{Pizo-F}. By proceeding with the same arguments as  in the proof of Theorem \ref{EXPONENTIAL-THM}, an exponential stability is obtained. 	
\end{rem}
%%%%%%%%%%%%%%%%%%%%%%%%%%%%%%%%%%%%%%%%%%%
%%%%%%%%%%%%%%%%%%%%%%%%%%%%%%%%%%%%%%%%%%%
%%%%%%%%%%%% New Section
%%%%%%%%%%%%%%%%%%%%%%%%%%%%%%%%%%%%%%%%%%%
\section{Piezoelectric with Gurtin-Pipkin Thermal law \eqref{Pizo-GP}}
\noindent In this section,    we shall analyze the strong stability and the polynomial stability of system \eqref{Pizo-GP}. 
\subsection{Strong Stability} In this subsection we will prove the stability of system \eqref{Pizo-GP}. The main result of this section is the following theorem. 
\begin{theoreme}\label{StrongStability}
Let $m=1$ and assume that \eqref{CONDITION-H} holds. Then, the $C_0-$semigroup of contractions $\left(e^{t\mathcal{A}_1}\right)_{t\geq 0}$ is strongly stable in $\mathcal{H}$; i.e., for all $U_0\in \mathcal{H}$, the solution of \eqref{evolution} satisfies 
$$
\lim_{t\to +\infty}\|e^{t\mathcal{A}_1}U_0\|_{\mathcal{H}}=0. 
$$ 	
\end{theoreme}
%%%%%%%%%%%%%%%
\noindent According to Theorem \ref{App-Theorem-A.2} in the appendix, to prove Theorem \ref{StrongStability}, we need to prove that the operator $\mathcal{A}_1$ has no pure imaginary eigenvalues and $\sigma(\mathcal{A}_1)\cap i\mathbb{R}$  is countable. The proof of Theorem \ref{StrongStability} will be achieved from the followig proposition.
\begin{pro}\label{iRrhoA}
Let $m=1$ and assume that \eqref{CONDITION-H} holds, we have 
\begin{equation}\label{iR}
i\mathbb{R}\subset \rho(\mathcal{A}_1).	
\end{equation}	
\end{pro}
\noindent We will prove Proposition \ref{iRrhoA} by contradiction argument. Remark that, it has been proved in Proposition \ref{m-dissipative} that $0\in \rho(\mathcal{A}_1)$. Now, suppose that \eqref{iR} is false, then there exists $l\in \mathbb{R}^{\ast}$ such that $i\,l \notin \rho(\mathcal{A}_1)$. According to Remark \ref{App-Lemma-A.3}, let $\left\{(\lambda^n,U^n:=(u^n,v^n,y^n,z^n,w^n,\eta^n)^{\top})\right\}_{n\geq 1}\subset \mathbb{R}^{\ast}\times D(\mathcal{A}_1)$, with 
\begin{equation}\label{arg-pol1}\tag{{\rm CA1}}
\la_n\to l\quad \text{as}\ n\to \infty \quad \text{and}\quad  \abs{\la_n}<\abs{l},
\end{equation}  
and 
\begin{equation}\label{U=1}
\|U^n\|_{\mathcal{H}}=\|\left(u^n,v^n,y^n,z^n,w^n,\eta^n\right)^{\top}\|_{\mathcal{H}}=1,	
\end{equation}
such that 
\begin{equation}\label{DEC-POL1}
 \left(i\la_nI-{\mathcal{A}}_1\right)U^n=F_n:=\left(f^1_n,f^2_n,f^3_n,f^4_n,f^5_n,f^6_n(\cdot,s)\right)\to 0\quad \text{in}\quad \mathcal{H},\quad \text{as}\ n\to \infty.	
 \end{equation}
Equivalently, from \eqref{DEC-POL1}, we have 
\begin{eqnarray}
i\lambda_n u^n-v^n&=&f^{1}_n\quad \text{in}\ \ H_L^1(0,L),\label{POL-eq1}\\
i\lambda_n\rho v^n-\alpha u^n_{xx}+\gamma\beta y^n_{xx}+\delta \omega^n_x&=&\rho f^{2}_n\quad \text{in}\ \ L^2(0,L),\label{POL-eq2}\\
i\lambda_n y^n-z^n&=&f^{3}_n\quad \text{in}\ \ H_L^1(0,L),\label{POL-eq3}\\ 
i\lambda_n\mu z^n-\beta y^n_{xx}+\gamma\beta u^n_{xx}&=&\mu f^{4}_n\quad \text{in}\ \ L^2(0,L),\label{POL-eq4}\\
i\lambda_n w^n-c\Lambda^{1,n}_{xx}+\delta v^n_x&=&f^{5}_n\quad \text{in}\ \ L^2(0,L),\label{POL-eq5}\\
i\lambda_n \eta^n+\eta^n_s-w^n&=&f^{6}_n(\cdot,s)\quad \text{in}\ \ W.\label{POL-eq6}
\end{eqnarray}
Then, we will proof condition \eqref{iR} by finding a contradiction with \eqref{U=1} such as $\|U^n\|_{\mathcal{H}}\to 0$. The proof of proposition \ref{iRrhoA} has been divided into several Lemmas. 
%%%%%%%%%%%%%%%%%%%%%%%%%%%%%%%%%%%%%%%%%
\begin{lemma}\label{POLSTA1}
Let $m=1$ and assume that \eqref{CONDITION-H} holds.  Then,  the solution $(u^n,v^n,y^n,z^n,w^n,\eta^n)\in D(\mathcal{A}_1)$ of \eqref{POL-eq1}-\eqref{POL-eq6} satisfies 
\begin{equation}\label{Lemma1-EQPOL1}
\displaystyle -\int_0^L\int_0^{\infty}\sigma'(s)\abs{\eta^n_x}^2dsdx \underset{n\to \infty}\longrightarrow 0\end{equation}
and 
\begin{equation}\label{Lemma1-EQPOL2}
\displaystyle \int_0^L\int_0^{\infty}\sigma(s)\abs{\eta^n_x}^2dsdx\underset{n\to \infty}\longrightarrow 0.
\end{equation}	
\end{lemma}
\begin{proof}
First, taking the inner product of \eqref{DEC-POL1} with $U^n$ in $\mathcal{H}$, we get 
$$
-\frac{c}{2}\int_0^L\int_0^{\infty}\sigma'(s)\abs{\eta^n_x}^2dsdx =-\Re\left(\left<{\mathcal{A}}_1U^n,U^n\right>_{\mathcal{H}}\right)\leq \|F_n\|_{\mathcal{H}}\|U^n\|_{\mathcal{H}}\underset{n\to \infty}\longrightarrow 0.	
$$
Then, \eqref{Lemma1-EQPOL1} holds. Using condition \eqref{CONDITION-H}, we get 
\begin{equation*}
\int_0^L\int_0^{\infty}\sigma(s)\abs{\eta^n_x(\cdot,s)}^2dsdx\leq -\frac{1}{d_{\sigma}}\int_0^L\int_0^{\infty}\sigma'(s)\abs{\eta^n_x(\cdot,s)}^2dsdx.	
\end{equation*}
Using \eqref{Lemma1-EQPOL1} in the above inequality, we get \eqref{Lemma1-EQPOL2}. The proof has been completed. 
\end{proof}
%%%%%%%%%%%%%%%%%%%%%%%%%%%%%%%%%%%%%%%%
%%% Lemma 2
%%%%%%%%%%%%%%%%%%%%%%%%%%%%%%%%%%%%%%%%
\begin{lemma}\label{POLSTA2}
Let $m=1$ and assume that \eqref{CONDITION-H} holds. Then, the solution $(u^n,v^n,y^n,z^n,w^n,\eta^n)\in D(\mathcal{A}_1)$ of \eqref{POL-eq1}-\eqref{POL-eq6} satisfies 
\begin{equation}\label{Lemma2-EQPOL1}
\int_0^L\abs{\omega^n_x}^2dx\underset{n\to \infty}\longrightarrow 0\quad \text{and}\quad 	\int_0^L\abs{\omega^n}^2dx\underset{n\to \infty}\longrightarrow 0 \end{equation}	
\end{lemma}
\begin{proof}The proof of this Lemma is divided into two steps.\\
\textbf{Step 1.} First, we prove the following estimation 
\begin{equation}\label{L2-S1-EQ1}
\begin{array}{l}
\displaystyle
\frac{g(0)}{2}\int_0^L\abs{w^n_x}^2dx\leq 2\abs{\la}^2\int_0^L\int_0^{\infty}\sigma(s)\abs{\eta^n_x}^2dsdx+\frac{\sigma(0)}{g(0)}\int_0^L\int_0^{\infty}(-\sigma'(s))\abs{\eta^n_x}^2dsdx\\
\displaystyle 
+2\int_0^L\int_0^{\infty}\sigma(s)\abs{(f_n^6)
_x}^2dsdx.
\end{array}
\end{equation}
From \eqref{POL-eq6}, we have 
\begin{equation}\label{L2-S1-EQ2}
\omega^n_x=i\la_n\eta^n_x+\eta^n_{sx}-(f^6_n)_x.
\end{equation}
Multiplying \eqref{L2-S1-EQ2} by $\sigma(s)\overline{w^n_x}$, integrating over $(0,\infty)\times (0,L)$, we get 
\begin{equation*}\label{L2-S1-EQ3}
g(0)\int_0^L\abs{w^n_x}^2dx=i\la \int_0^L\int_0^{\infty}\sigma(s)\eta^n_x\overline{w^n_x}dsdx+\int_0^L\int_0^{\infty}\sigma(s)\eta^n_{sx}\overline{w^n_x}dsdx-\int_0^L\int_0^{\infty}\sigma(s)(f_n^6)_x\overline{w^n_x}dsdx.
\end{equation*}
Using integration by parts with respect to $s$ in the above equation and the fact that $\eta^n(\cdot,0)=0$ in $(0,L)$, we get 
\begin{equation}\label{L2-S1-EQ3}
g(0)\int_0^L\abs{w^n_x}^2dx=i\la \int_0^L\int_0^{\infty}\sigma(s)\eta^n_x\overline{w^n_x}dsdx-\int_0^L\int_0^{\infty}\sigma'(s)\eta^n_{x}\overline{w^n_x}dsdx-\int_0^L\int_0^{\infty}\sigma(s)(f_n^6)_x\overline{w^n_x}dsdx.	
\end{equation}
It follows that, 
\begin{equation}\label{L2-S1-EQ4}
\begin{array}{l}
\displaystyle 
g(0)\int_0^L\abs{w^n_x}^2dx\leq \abs{\la}\int_0^L\int_0^{\infty}\sigma(s)\abs{\eta^n_x}\abs{w^n_x}dsdx+\int_0^L\int_0^{\infty}-\sigma'(s)\abs{\eta^n_{x}}\abs{w^n_x}dsdx\\[0.1in]
\displaystyle 
+\int_0^L\int_0^{\infty}\sigma(s)\abs{(f_n^6)_x}\abs{w^n_x}dsdx.	
\end{array}
\end{equation}
Applying Cauchy-Schwarz and Young's inequality, we get 
\begin{equation}\label{L2-S1-EQ5}
\abs{\la}\int_0^L\int_0^{\infty}\sigma(s)\abs{\eta^n_x}\abs{w^n_x}dxds\leq \frac{\abs{\la}^2}{2\varepsilon}\int_0^L\int_0^{\infty}\sigma(s)\abs{\eta^n_x}^2dsdx+\frac{\varepsilon g(0)}{2}\int_0^L\abs{w^n_x}^2dx,	
\end{equation}
\begin{equation}\label{L2-S1-EQ6}
\begin{array}{rll}
\displaystyle \int_0^L\int_0^{\infty}-\sigma'(s)\abs{\eta^n_{x}}\abs{w^n_x}dsdx&\leq & \displaystyle 
\sqrt{\sigma(0)}\left(\int_0^L\abs{w^n_x}^2dx\right)^{\frac{1}{2}}\left(\int_0^L\int_0^{\infty}-\sigma'(s)\abs{\eta^n_x}^2dsdx\right)^{\frac{1}{2}}\\[0.1in]
&\leq &\displaystyle 
\frac{\varepsilon_1}{2}\int_0^L\abs{w^n_x}^2+\frac{\sigma(0)}{2\varepsilon_1}\int_0^L\int_0^{\infty}-\sigma'(s)\abs{\eta^n_x}^2dsdx
\end{array}	
\end{equation}
and 
\begin{equation}\label{L2-S1-EQ7}
\int_0^L\int_0^{\infty}\sigma(s)\abs{(f_n^6)_x}\abs{w^n_x}dsdx\leq \frac{\varepsilon_2 g(0)}{2}\int_0^L\abs{w^n_x}^2dxds+\frac{1}{2\varepsilon_2}\int_0^L\int_0^{\infty}\sigma(s)\abs{(f_n^6)_x}dsdx.
\end{equation}
Inserting \eqref{L2-S1-EQ5}-\eqref{L2-S1-EQ7} in \eqref{L2-S1-EQ4}, we get 
\begin{equation}\label{L2-S1-EQ8}
\begin{array}{l}
\displaystyle 
\left(g(0)-\frac{\varepsilon g(0)}{2}-\frac{\varepsilon_1}{2}-\frac{\varepsilon_2 g(0)}{2}\right)\int_0^L\abs{w^n_x}^2dx\leq \frac{\abs{\la}^2}{2\varepsilon}\int_0^L\int_0^{\infty}\sigma(s)\abs{\eta^n_x}^2dsdx+\\[0.1in]
\displaystyle
\frac{\sigma(0)}{2\varepsilon_1}\int_0^L\int_0^{\infty}-\sigma'(s)\abs{\eta^n_x}^2dsdx+\frac{1}{2\varepsilon_2}\int_0^L\int_0^{\infty}\sigma(s)\abs{(f_n^6)_x}dsdx.
\end{array}	
\end{equation}
Taking $\varepsilon=\varepsilon_2=\frac{1}{4}$ and $\varepsilon_1=\frac{g(0)}{2}$ in \eqref{L2-S1-EQ8}, we get \eqref{L2-S1-EQ1}.\\
\textbf{Step 2.} The aim of this step is to prove \eqref{Lemma2-EQPOL1}. For this aim, using Lemma \eqref{POLSTA1} and the fact that $\|f_6^n\|_{W}\to 0$ in $\mathcal{H}$ in \eqref{L2-S1-EQ1}, we get  the first estimation in \eqref{Lemma2-EQPOL1}. Next, using Poincar\'e inequality, we get 
$$
\int_0^L\abs{w^n}^2dx\leq c_p\int_0^L\abs{w^n_x}^2dx\underset{n\to \infty}\longrightarrow 0.
$$
The proof has been completed.  
\end{proof}
%%%%%%%%%%%%%%%%%%%%%%%%%%%%%%%%%%%%%%%%%%%%%%
%%%%%%%%%%%%%%%%%%%%%%%%%%%%%%%%%%%%%%%%%%%%%%
\begin{lemma}\label{POLSTA3}
Let $m=1$ and assume that \eqref{CONDITION-H} holds. Then, the solution $(u^n,v^n,y^n,z^n,w^n,\eta^n)\in D(\mathcal{A}_1)$ of \eqref{POL-eq1}-\eqref{POL-eq6} satisfies
\begin{equation}\label{0Lemma3-EQPOL1}
\int_0^L\abs{u^n_x}^2dx\underset{n\to \infty}\longrightarrow 0\end{equation}
and 
\begin{equation}\label{Lemma3-EQPOL2}
\int_0^L\abs{v^n}^2dx\underset{n\to \infty}\longrightarrow  0.
\end{equation}
\end{lemma}
\begin{proof}
The proof of this Lemma is divided into several steps.\\
\textbf{Step 1.} The aim of this step is to prove the following estimations:
\begin{equation}\label{L3-S1-EQ1}
\|\Lambda_{xx}^{1,n}\|_{L^2(0,L)}\lesssim \abs{\lambda_n}\|U^n\|_{\mathcal{H}}+\|F^n\|_{\mathcal{H}}.	
\end{equation}
\begin{equation}\label{L3-S1-EQ2}
\|u_{xx}^n\|_{L^2(0,L)}\lesssim \abs{\lambda_n}\|U^n\|_{\mathcal{H}}+\|w^n_x\|_{L^2(0,L)}+\|F^n\|_{\mathcal{H}}.	
\end{equation}
First, we prove \eqref{L3-S1-EQ1}. Using \eqref{POL-eq5} and \eqref{POL-eq1}, we get 
$$
c\Lambda_{xx}^{1,n}=i\lambda_nw^n+\delta i\lambda_n u_x^n-\delta(f_n^1)_x-f_n^5. 
$$
Using the fact that $\|w^n\|\leq \|U^n\|_{\mathcal{H}}$, $\sqrt{\alpha_1}\|u^n_x\|\leq \|U^n\|_{\mathcal{H}}$, $\|f^5_n\|\leq \|F_n\|_{\mathcal{H}}$ and $\sqrt{\alpha_1}\|(f_n^1)_x\|\leq \|F^n\|_{\mathcal{H}}$, we get 
$$
c\|\Lambda_{xx}^{1,n}\|_{L^2(0,L)}\leq \left(1+\frac{\delta}{\sqrt{\alpha_1}}\right)\left(\abs{\lambda_n}\|U^n\|_{\mathcal{H}}+\|F_n\|_{\mathcal{H}}\right).
$$
Then, we get \eqref{L3-S1-EQ1}. In order to prove \eqref{L3-S1-EQ2}, using \eqref{POL-eq2} and the fact that $\alpha=\alpha_1+\gamma^2\beta$, we get 
$$
\alpha_1u^n_{xx}=i\lambda\left(\rho v^n+\gamma \mu z^n\right)+\delta w^n_x-\rho f^2_n-\gamma \mu f^4_n.
$$
Using the fact that $\sqrt{\rho}\|v^n\|\leq \|U^n\|_{\mathcal{H}}$, $\sqrt{\mu}\|z^n\|\leq \|U^n\|_{\mathcal{H}}$, $\sqrt{\rho}\|f^2_n\|\leq \|F^n\|_{\mathcal{H}}$ and $\sqrt{\mu}\|f_4^n\|\leq \|F^n\|_{\mathcal{H}}$ in the above inequality, we obtain 
\begin{equation*}
\alpha_1\|u^n_{xx}\|_{L^2(0,L)}\leq\left(\sqrt{\rho}+\gamma\sqrt{\mu}\right)\left(\abs{\lambda_n}\|U^n\|_{\mathcal{H}}+\|F_n\|_{\mathcal{H}}\right)+\delta \|w^n_x\|.
\end{equation*}
Then, we get \eqref{L3-S1-EQ2}.\\
\textbf{Step 2.} The aim of this step is to prove \eqref{0Lemma3-EQPOL1}. From \eqref{POL-eq1} and \eqref{POL-eq5}, we have 
\begin{equation*}
i\la_n\delta u^n_x=-i\la_n w^n+c\Lambda_{xx}^{1,n}+\delta (f^1_n)_x+f^5.	
\end{equation*}
Multiplying the above equation by $-i\la_n^{-1}\delta \overline{u^n_x}$, integrating by parts over $(0,L)$, we get 
\begin{equation*}
\delta \int_0^L\abs{u^n_x}^2dx=-\delta\int_0^Lw^n\overline{u_x^n}dx+i\la^{-1}c\int_0^L\Lambda^{1,n}_x\overline{u^n_{xx}}dx+i\la_n^{-1}c\Lambda^n_x(0)\overline{u_x^n(0)}-i\la_n^{-1}\int_0^L(\delta (f^1_n)_x+f^5_n)\overline{u^n_x}dx.	
\end{equation*}
It follows that, 
\begin{equation}\label{L3-S2-EQ1}
\begin{array}{l}
\displaystyle
\delta \int_0^L\abs{u^n_x}^2dx\leq \delta \int_0^L\abs{w^n}\abs{u^n_x}dx+\abs{\la_n}^{-1}c\int_0^L\abs{\Lambda_x^{1,n}}\abs{u^n_{xx}}dx	+c\abs{\lambda_n^{-1}}\abs{\Lambda_x^n(0)}\abs{u^n_x(0)}\\[0.1in]
\displaystyle 
\abs{\la_n^{-1}}\delta\int_0^L\abs{(f_n^1)_x}\abs{u^n_x}dx+\abs{\la_n^{-1}}\int_0^L\abs{f^5_n}\abs{u^n_x}dx. 
\end{array}	
\end{equation}
Using the fact that $\sqrt{\alpha_1}\|u^n_x\|\leq \|U^n\|_{\mathcal{H}}$, \eqref{Lemma2-EQPOL1}, we get 
\begin{equation}\label{L3-S2-EQ2}
 \int_0^L\abs{w^n}\abs{u^n_x}dx\underset{n\to \infty}\longrightarrow 0.
 \end{equation}
Using the fact that $\sqrt{\alpha_1}\|(f^1_n)_x\|\leq \|F_n\|_{\mathcal{H}}\to 0$, $\|f^5_n\|\leq \|F_n\|_{\mathcal{H}}\to 0$,$\sqrt{\alpha_1}\|u^n_x\|\leq \|U^n\|_{\mathcal{H}}$  and \eqref{arg-pol1} we get
\begin{equation}\label{L3-S2-EQ3}
\abs{\la_n^{-1}}\delta\int_0^L\abs{(f_n^1)_x}\abs{u^n_x}dx\underset{n\to \infty}\longrightarrow 0.
\end{equation}
and 

\begin{equation}\label{L3-S2-EQ4}
\abs{\la_n^{-1}}\int_0^L\abs{f^5_n}\abs{u^n_x}dx\underset{n\to \infty}\longrightarrow 0.
\end{equation}
Using \eqref{Lemma1-EQPOL2}, \eqref{Lemma2-EQPOL1}, \eqref{L3-S1-EQ2}, \eqref{U=1}  and \eqref{arg-pol1},   we get 
\begin{equation}\label{L3-S2-EQ5}
\abs{\la_n}^{-1}c\int_0^L\abs{\Lambda_x^{1,n}}\abs{u^n_{xx}}dx\underset{n\to \infty}\longrightarrow 0.
\end{equation}
Using Gagliardo-Nirenberg inequality, \eqref{L3-S1-EQ1}, \eqref{L3-S1-EQ2}, $\|U^n\|_{\mathcal{H}}=1$, $\|F_n\|_{\mathcal{H}}\to 0$ and \eqref{arg-pol1} we get 
\begin{equation}\label{L3-S2-EQ6}
 \abs{\Lambda^{1,n}_x(0)}\lesssim \left(\|\Lambda^{1,n}_{xx}\|^{\frac{1}{2}}\|\Lambda^{1,n}_{x}\|^{\frac{1}{2}}+\|\Lambda^{1,n}_{x}\|\right)\underset{n\to \infty}\longrightarrow 0
 \end{equation}
and 
\begin{equation}\label{L3-S2-EQ7}
\abs{u^n_x(0)}\lesssim \left(\|u^n_{xx}\|^{\frac{1}{2}}\|u_x^n\|^{\frac{1}{2}}+\|u_x^n\|\right) \lesssim M.
\end{equation}
From \eqref{L3-S2-EQ6} and \eqref{L3-S2-EQ7}, we obtain 
\begin{equation}\label{L3-S2-EQ8}
\abs{\lambda_n^{-1}}\abs{\Lambda_x^n(0)\abs{u^n_x(0)}}\underset{n\to \infty}\longrightarrow 0.
\end{equation}
Finally, inserting \eqref{L3-S2-EQ2}, \eqref{L3-S2-EQ3}, \eqref{L3-S2-EQ4}, \eqref{L3-S2-EQ5} and \eqref{L3-S2-EQ8} in \eqref{L3-S2-EQ1}, we get the desired result \eqref{0Lemma3-EQPOL1}.\\
\textbf{Step 3.} The aim of this step is to prove \eqref{Lemma3-EQPOL2}. From \eqref{0Lemma3-EQPOL1} and Poincar\'e Inequality, we get 
\begin{equation}\label{L3-S3-EQ1}
\int_0^L\abs{u^n}^2dx\underset{n\to \infty}\longrightarrow 0.
\end{equation}
From \eqref{POL-eq1}, we get 
\begin{equation*}
\int_0^L\abs{v^n}^2dx\leq 2\abs{\la_n}^2\int_0^L\abs{u^n}^2dx+2c_p\int_0^L\abs{(f_n^1)_x}^2dx.	
\end{equation*}
Passing to the limit in the above inequality and  using \eqref{arg-pol1}, \eqref{L3-S3-EQ1} and the fact that $\sqrt{\alpha_1}\|(f^1_n)_x\|\leq \|F^n\|_{\mathcal{H}}\underset{n\to \infty}\longrightarrow 0$, we get the desired result \eqref{Lemma3-EQPOL2}. The proof is thus completed.  
\end{proof}

%%%%%%%%%%%%%%
\noindent Inserting \eqref{POL-eq1} in \eqref{POL-eq2}, we get 
\begin{equation}\label{Combining}
-\la^2\rho u^n-\alpha u^n_{xx}+\gamma \beta y^n_{xx}+\delta w^n_x=\rho f^2_n+i\lambda_n\rho f^1_n.	
\end{equation}
%%%%%%%%%%%%%%%%%%%

\begin{lemma}\label{POLSTA5}
Let $m=1$ and assume that \eqref{CONDITION-H} holds. Then, the solution $(u^n,v^n,y^n,z^n,w^n,\eta^n)\in D(\mathcal{A}_1)$ of \eqref{POL-eq1}-\eqref{POL-eq6} satisfies the following estimation 
\begin{equation}\label{Lemma5-EQPOL1}
\int_0^L\abs{y^n_x}^2dx\underset{n\to \infty}\longrightarrow 0
\end{equation}
and 
\begin{equation}\label{Lemma5-EQPOL2}
\int_0^L\abs{z^n}^2dx\underset{n\to \infty}\longrightarrow  0.
\end{equation} 	
\end{lemma}
\begin{proof}
The proof of this Lemma is divided into two steps.\\
\textbf{Step 1.} The aim of this step is to prove \eqref{Lemma5-EQPOL1}. For this aim, Multiplying \eqref{Combining} by $-\overline{y}$ and integrating by parts over $(0,L)$, we get 
\begin{equation*}\label{Lemma5-EQPOL3}
\gamma\beta \int_0^L\abs{y^n_x}^2dx=\rho\la_n^2\int_0^Lu\overline{y}dx+\alpha\int_0^Lu^n_x\overline{y^n_x}dx-\delta\int_0^Lw^n\overline{y^n_x}dx-\int_0^L\left(\rho f^2_n+i\lambda_n\rho f^1_n\right)\overline{y^n}dx.	
\end{equation*}
It follows that, 
\begin{equation}\label{L5-S1-EQ1}
\gamma\beta \int_0^L\abs{y^n_x}^2dx\leq \rho\la_n^2\int_0^L\abs{u}\abs{y}dx+\alpha\int_0^L\abs{u^n_x}\abs{y^n_x}dx+\delta\int_0^L\abs{w^n}\abs{y^n_x}dx+\int_0^L\left(\rho\abs{f_n^2}+\rho \abs{\la_n}\abs{f^1_n}\right)\abs{y^n}dx. 	
\end{equation}
Using the fact that $\la_ny^n$ is uniformly bounded in $L^2(0,L)$, \eqref{L3-S3-EQ1} and \eqref{arg-pol1}, we get 
\begin{equation}\label{L5-S1-EQ2}
\la_n^2\int_0^L\abs{u}\abs{y}dx\underset{n\to \infty}\longrightarrow 0.
\end{equation}
Using the fact that $y^n_x$ is uniformly bounded in $L^2(0,L)$ and \eqref{0Lemma3-EQPOL1}, we get 
\begin{equation}\label{L5-S1-EQ3}
\int_0^L\abs{u^n_x}\abs{y^n_x}dx\underset{n\to \infty}\longrightarrow 0.
\end{equation}
Using the fact that $y^n_x$ is uniformly bounded in $L^2(0,L)$ and \eqref{Lemma2-EQPOL1}, we get 
\begin{equation}\label{L5-S1-EQ4}
\int_0^L\abs{w^n}\abs{y^n_x}dx\underset{n\to \infty}\longrightarrow 0.
\end{equation}
Using the fact that $\|F_n\|_{\mathcal{H}}\to 0$ and the fact that $\la_ny^n$ is uniformly bounded in $L^2(0,L)$, we get 
\begin{equation}\label{L5-S1-EQ5}
\int_0^L\left(\rho\abs{f_n^2}+\rho \abs{\la_n}\abs{f^1_n}\right)\abs{y^n}dx\underset{n\to \infty}\longrightarrow 0.
\end{equation}  
Inserting \eqref{L5-S1-EQ2}-\eqref{L5-S1-EQ5} in \eqref{L5-S1-EQ1}, we obtain \eqref{Lemma5-EQPOL1}.\\
\textbf{Step 2.} The aim of this step is to prove \eqref{Lemma5-EQPOL2}. From \eqref{Lemma5-EQPOL1} and Poincar\'e inequality, we get 
\begin{equation}\label{L5-S2-EQ1}
\int_0^L\abs{y^n}^2dx\underset{n\to \infty}\longrightarrow 0.
\end{equation}
From \eqref{POL-eq3}, we get 
\begin{equation*}
\int_0^L\abs{z^n}^2dx\leq 2\abs{\la_n}^2\int_0^L\abs{y^n}^2dx+2c_p\int_0^L\abs{(f_n^3)_x}^2dx.	
\end{equation*}
Passing to the limit in the above inequality and  using \eqref{arg-pol1}, \eqref{L5-S2-EQ1} and the fact that $\|F_n\|_{\mathcal{H}}\to 0$, we get the desired result \eqref{Lemma5-EQPOL2}. The proof is thus completed.  
\end{proof}
%%%%%%%%%%%%%%%%%%%%%%%%%%%%%%%%%%%%%%%%%
%%%%%%%%%%%%%%%%%%%%%%%%%%%%%%%%%%%%%%%%%
\\
\\
\noindent \textbf{Proof of Proposition \ref{iRrhoA}.} From Lemmas \ref{POLSTA1}-\ref{POLSTA5}, we obtain $\|U^n\|_{\mathcal{H}}\to 0$ as $n\to +\infty$ which contradicts $\|U^n\|_{\mathcal{H}}=1$. Thus, \eqref{iR} is holds true. The proof is thus complete.\xqed{$\square$}
\\
\\
 \noindent \textbf{Proof of Theorem \ref{StrongStability}.} From Proposition \ref{iRrhoA}, we have $i\mathbb{R}\subset \rho(\mathcal{A})$ and consequently $\sigma(\mathcal{A})\cap i\mathbb{R}=\emptyset$. Therefore, according to Theorem \ref{App-Theorem-A.2} in the appendix, we get the $C_0-$semigroup of contraction $\left(e^{t\mathcal{A}_1}\right)_{t\geq 0}$ is strongly stable. The proof is thus complete.
 %%%%%%%%%%%%%%%%%%%%%%%%%%%%%%%%%%% 
 \subsection{Polynomial Stability}  
In this subsection, we will prove the polynomial stability of  system \eqref{Pizo-GP}. The main result of this section is the following theorem. 
%%%%%%%%%%%%%%%%%%
\begin{theoreme}\label{POLY-STA-m=1}
Let $m=1$ and assume that \eqref{CONDITION-H} holds, then there exists $C>0$ such that for every $U_{0}\in D(\AA_1)$, we have 
	\begin{equation}\label{2}
		E_1(t)\leq \frac{C}{t}\|U_{0}\|^2_{D({\AA}_1)},\quad t>0.
		\end{equation}	
\end{theoreme}
\noindent According to Theorem \ref{bt} in appendix, to prove Theorem \ref{POLY-STA-m=1}, we still need to prove the following two conditions 
\begin{equation}\label{COND-POL1}\tag{{\rm POL1}}
i\mathbb{R}\subset \rho(\mathcal{A}_1),	
\end{equation}
\begin{equation}\label{COND-POL2}\tag{{\rm POL2}}
\limsup_{\abs{\lambda}\to \infty} \frac{1}{\abs{\la}^{2}}\|(i\lambda I-\mathcal{A}_1)^{-1}\|<\infty.
\end{equation}
From Proposition \ref{iRrhoA}, we obtain condition \eqref{COND-POL1}. Next, we will prove condition \eqref{COND-POL2} by a contradiction argument. For this purpose, suppose that \eqref{COND-POL2} is false, then there exists $\left\{(\lambda^n,U^n:=(u^n,v^n,y^n,z^n,w^n,\eta^n(\cdot,s)))\right\}_{n\geq 1}\subset \mathbb{R}^{\ast}\times D(\mathcal{A}_1)$ with
\begin{equation}\label{arg-pol2}\tag{{\rm CA2}}
\abs{\la_n}\to \infty\quad \text{and}\quad \|U^n\|_{\mathcal{H}}=\|(u^n,v^n,y^n,z^n,w^n,\eta^n(\cdot,s))\|_{\mathcal{H}}=1,	
\end{equation}
such that 
\begin{equation}\label{Pol-Comp}
\lambda_n^2\left(i\la_nI-{\mathcal{A}}_1\right)U^n=F_n:=\left(f^1_n,f^2_n,f^3_n,f^4_n,f^5_n,f^6_n(\cdot,s)\right)^{\top}\to 0\quad \text{in}\quad \mathcal{H}.		
\end{equation}
For simplicity, we drop the index $n$. Equivalently, from \eqref{Pol-Comp}, we have 
\begin{eqnarray}
i\lambda u-v&=&\lambda^{-2}f^1\quad \text{in}\ \ H_L^1(0,L),\label{1POL-eq1}\\
i\lambda\rho v-\alpha u_{xx}+\gamma\beta y_{xx}+\delta \omega_x&=&\rho \lambda^{-2}f^2\quad \text{in}\ \ L^2(0,L),\label{1POL-eq2}\\
i\lambda y-z&=&\lambda^{-2}f^3\quad \text{in}\ \ H_L^1(0,L),\label{1POL-eq3}\\ 
i\lambda\mu z-\beta y_{xx}+\gamma\beta u_{xx}&=&\mu \lambda^{-2}f^4\quad \text{in}\ \ L^2(0,L),\label{1POL-eq4}\\
i\lambda w-c\Lambda^{1}_{xx}+\delta v_x&=&\lambda^{-2}f^5\quad \text{in}\ \ L^2(0,L),\label{1POL-eq5}\\
i\lambda \eta+\eta_s-w&=&\lambda^{-2}f^6(\cdot,s)\quad \text{in}\ \ W.\label{1POL-eq6}
\end{eqnarray}
%%%%%%%%%%%%%%%%%
Here, we will check the condition \eqref{COND-POL2} by finding a contradiction with \eqref{arg-pol2} such as $\|U\|_{\mathcal{H}}=o(1)$. For clarity, we divide the proof into several lemmas. 
%%%%%%%%%%%%%%%%%
%%%%%%%%%%%%%%%%%%%%%%%%%%%%%%%%%%%%%%%%%
%%% Lemma 
%%%%%%%%%%%%%%%%%%%%%%%%%%%%%%%%%%%%%%%%
%

%%%%%%%%%%%%%%%%%%%%%%%%%%%%%%%%%%%%%%%%%%%%%%%
%%%%%%%%%%%%%%%%%%%%%%%%%%%%%%%%%%%%%%%%%%%%%%
\begin{lemma}\label{EST1-POL1}
Let $m=1$ and assume that \eqref{CONDITION-H} holds. The solution $(u,v,y,z,w,\eta)\in D(\mathcal{A}_1)$ of \eqref{1POL-eq1}-\eqref{1POL-eq6} satisfies the following estimations 
\begin{equation}\label{1EST1-POL1}
-\int_0^L\int_0^{\infty}\sigma'(s)\abs{\eta_x}^2dsdx=o(\la^{-2})\quad \text{and}\quad \int_0^L\int_0^{\infty}\abs{\eta_x}^2dsdx=o(\la^{-2}).	
\end{equation}	
\end{lemma}
\begin{proof}
By proceeding the same argument used in Lemma \ref{POLSTA1}, we get \eqref{1EST1-POL1}
.\end{proof}
%%%%%%%%%%%%%%%%%%%%%%%%%%%%%%%%%%%%%%%%%%%%%%%
%%%%%%%%%%%%%%%%%%%%%%%%%%%%%%%%%%%%%%%%%%%%%%
\begin{lemma}\label{0POLSTA3}
Let $m=1$ and assume that \eqref{CONDITION-H} holds. The solution $(u,v,y,z,w,\eta)\in D(\mathcal{A}_1)$ of \eqref{1POL-eq1}-\eqref{1POL-eq6} satisfies  the following estimations 
\begin{equation}\label{Lemma3-EQPOL1}
\int_0^L\abs{w_x}^2dx =o(1)\quad \text{and}\quad \int_0^L\abs{w}^2dx =o(1).
\end{equation}
\end{lemma}
\begin{proof}
By using the same techniques  of step 1 in Lemma \ref{POLSTA2}, we get 	
\begin{equation}\label{Pol-L2-S1-EQ1}
\begin{array}{l}
\displaystyle
\frac{g(0)}{2}\int_0^L\abs{w_x}^2dx\leq 2\abs{\la}^2\int_0^L\int_0^{\infty}\sigma(s)\abs{\eta_x}^2dsdx+\frac{\sigma(0)}{g(0)}\int_0^L\int_0^{\infty}(-\sigma'(s))\abs{\eta_x}^2dsdx\\
\displaystyle 
+2\abs{\la}^{-4}\int_0^L\int_0^{\infty}\sigma(s)\abs{(f^6)
_x}^2dsdx.
\end{array}
\end{equation}
Using the fact that $\|F\|_{\mathcal{H}}=o(1)$, Lemma \ref{EST1-POL1} in \eqref{Pol-L2-S1-EQ1}, we get the first estimation in \eqref{Pol-L2-S1-EQ1}. Applying Poincar\'e inequality, we get the second estimation in \eqref{Pol-L2-S1-EQ1}. The proof has been completed.  
\end{proof}
%%%%%%%%%%%%%%%%%%%%%%%%%%%%%%%%%%%%%%%%%%%%%%%%%%
\begin{lemma}\label{1POLSTA3}
Let $m=1$ and assume that \eqref{CONDITION-H} holds. The solution $(u,v,y,z,w,\eta)\in D(\mathcal{A}_1)$ of \eqref{1POL-eq1}-\eqref{1POL-eq6} satisfies  the following estimation
\begin{equation}\label{Lemma4-EQPOL1}
\int_0^L\abs{u_x}^2dx=o(1). 	
\end{equation}
\end{lemma}
\begin{proof}
Similar to Lemma \ref{POLSTA3}, we 	prove that 
\begin{equation}\label{1Lemma4-EQPOL1}
\left\{\begin{array}{l}
\|\Lambda_{xx}^{1}\|_{L^2(0,L)}\lesssim \abs{\lambda}\|U\|_{\mathcal{H}}+\|F\|_{\mathcal{H}}\leq O(\abs{\lambda}),\\[0.1in]
\|u_{xx}\|_{L^2(0,L)}\lesssim \abs{\lambda}\|U\|_{\mathcal{H}}+\|w_x\|_{L^2(0,L)}+\|F\|_{\mathcal{H}}\leq O(\abs{\lambda}). 	
\end{array}
\right.
\end{equation}
Now, multiplying \eqref{1POL-eq4}  by $-i\la^{-1}\delta \overline{u_x}$, integrating by parts over $(0,L)$ and using \eqref{1POL-eq1}, we get 
\begin{equation}\label{2Lemma4-EQPOL1}
\delta \int_0^L\abs{u_x}^2dx=-\delta\int_0^Lw\overline{u_x}dx+i\la^{-1}c\int_0^L\Lambda^{1}_x\overline{u_{xx}}dx+i\la^{-1}c\Lambda_x(0)\overline{u_x(0)}-i\la^{-3}\int_0^L(\delta (f^1_n)_x+f^5_n)\overline{u^n_x}dx.	
\end{equation}
Using the facts that $u_x$ is uniformly bounded in $L^2(0,L)$, Lemma \ref{EST1-POL1}, \eqref{Lemma3-EQPOL1}, \eqref{1Lemma4-EQPOL1} and $\|F\|_{\mathcal{H}}=o(1)$ , we get 
\begin{equation}\label{1POLSTA31}
\left|\int_0^Lw\overline{u_x}dx\right|=o(1),\quad \left|\la^{-1}\int_0^L\Lambda^{1}_x\overline{u_{xx}}dx\right|=o(\abs{\la}^{-1})\ \quad\text{and}\quad \left|\la^{-3}\int_0^L(\delta (f^1_n)_x+f^5_n)\overline{u^n_x}dx\right|=o(\abs{\la}^{-3}). 
\end{equation}
 Using Gagliardo-Nirenberg inequality, \eqref{1Lemma4-EQPOL1}, \eqref{1EST1-POL1} and the fact that $u_x$ is uniformly bounded in $L^2(0,L)$,  we get 
 \begin{equation}\label{1POLSTA32}
\left|\la^{-1}\Lambda_x^1(0)u_x(0)\right|\lesssim \abs{\la}^{-1}\underbrace{\left(\|\Lambda_{xx}^{1}\|^{\frac{1}{2}}\|\Lambda^1_x\|^{\frac{1}{2}}+\|\Lambda^1_x\|\right)}_{o(\abs{\lambda}^{\frac{1}{2}})}\underbrace{\left(\|u_{xx}\|^{\frac{1}{2}}\|u_x\|^{\frac{1}{2}}+\|u_x\|\right)}_{O(\abs{\la}^{\frac{1}{2}})}=o(1).	
 \end{equation}
 Inserting \eqref{1POLSTA31} and \eqref{1POLSTA32} in \eqref{2Lemma4-EQPOL1}, we get the desired result \eqref{Lemma4-EQPOL1}. The proof has been completed.  
\end{proof}
%%%%%%%%%%%%%%%%%%%%%%%%%%%%%%%%%%%%%%%%%%%%%%%%%%
\\

\noindent Inserting \eqref{1POL-eq1} in \eqref{1POL-eq2}, we get 
\begin{equation}\label{NEWCombining}
-\la^2\rho u-\alpha u_{xx}+\gamma \beta y_{xx}+\delta w_x=\rho \la^{-2}f^2+i\la^{-1}\rho f^1.	
\end{equation}

\begin{lemma}\label{POLSTA4}
Let $m=1$ and assume that \eqref{CONDITION-H} holds. The solution $(u,v,y,z,w,\eta)\in D(\mathcal{A}_1)$ of \eqref{1POL-eq1}-\eqref{1POL-eq6} satisfies  the following estimation
\begin{equation}\label{1Lemma4-EQ1}
\int_0^L\abs{\la u}^2dx=o(1).	
\end{equation}
\end{lemma}
\begin{proof}
Multiplying \eqref{NEWCombining} by $-\overline{u}$ integrating by parts over $(0,L)$, we get 
\begin{equation*}
\rho\int_0^L\abs{\la u}^2dx=\alpha \int_0^L\abs{u_x}^2dx-\gamma \beta \int_0^Ly_x\overline{u_x}dx+\delta \int_0^Lw_x\overline{u}dx-\int_0^L\left(\frac{\rho f^2}{\la^2}+\frac{i\rho f^1}{\la}\right)\overline{u}dx.	
\end{equation*}
It follows that 
\begin{equation}\label{L4-eq1}
\rho\int_0^L\abs{\la u}^2dx\leq \alpha \int_0^L\abs{u_x}^2dx+\gamma\beta \int_0^L\abs{y_x}\abs{u_x}dx+\delta\int_0^L\abs{w_x}\abs{u}dx+\int_0^L\left(\frac{\rho\abs{f^2}}{\la^2}+\frac{\rho \abs{f^1}}{\abs{\lambda}}\right)\abs{u}dx.
\end{equation}
Using the fact that $y_x$ and $\la u$ are uniformly bounded in $L^2(0,L)$, \eqref{Lemma3-EQPOL1}, we get 
\begin{equation}\label{L4-eq2}
\int_0^L\abs{y_x}\abs{u_x}dx=o(1)\quad \text{and}\quad \int_0^L\abs{w_x}\abs{u}dx=\frac{o(1)}{\la}.	
\end{equation}
Using the fact that $\|F\|_{\mathcal{H}}=o(1)$ and $\la u$ is uniformly bounded in $L^2(0,L)$, we get 
\begin{equation}\label{L4-eq3}
\int_0^L\left(\frac{\rho\abs{f^2}}{\la^2}+\frac{\rho \abs{f^1}}{\abs{\lambda}}\right)\abs{u}dx=\frac{o(1)}{\la^2}. 	
\end{equation}
Inserting \eqref{L4-eq2} and \eqref{L4-eq3} in \eqref{L4-eq1} and using \eqref{Lemma3-EQPOL1}, we obtain \eqref{1Lemma4-EQ1}. The proof is thus completed. 
\end{proof}
%%%%%%%%%%%%%%%%%%%%%%%%%%%%%%%%%%%%
\begin{lemma}\label{poly-yx}
Let $m=1$ and assume that \eqref{CONDITION-H} holds. The solution $(u,v,y,z,w,\eta)\in D(\mathcal{A}_1)$ of \eqref{1POL-eq1}-\eqref{1POL-eq6} satisfies  the following estimation
\begin{equation}\label{1poly-yx}
\int_0^L\abs{y_x}^2dx=o(1).	
\end{equation}
\end{lemma}
\begin{proof}
The idea of proof is similar to Lemma \ref{POLSTA5}. Multiplying \eqref{NEWCombining} by $-\overline{y}$,  integrating by parts over $(0,L)$ and using the facts that $\la y$ and $y_x$ are uniformly bounded in $L^2(0,L)$, \eqref{1Lemma4-EQ1}, \eqref{Lemma4-EQPOL1} and \eqref{Lemma3-EQPOL1} and $\|F\|_{\mathcal{H}}=o(1)$, we get 
 \begin{equation*}
\gamma\beta \int_0^L\abs{y_x}^2dx\leq \underbrace{\rho\la^2\int_0^L\abs{u}\abs{y}dx}_{o(1)}+\underbrace{\alpha\int_0^L\abs{u^n_x}\abs{y^n_x}dx}_{o(1)}+\underbrace{\delta\int_0^L\abs{w^n}\abs{y^n_x}dx}_{o(1)}+\underbrace{\int_0^L\left(\rho\abs{\la}^{-2}\abs{f^2}+\rho \abs{\la}^{-1}\abs{f^1}\right)\abs{y}dx}_{o(\abs{\la}^{-2})}. 	
\end{equation*}  
The proof has been completed.
\end{proof}
%%%%%%%%%%%%%%%%%%%%%%%%%%%%%%%%%%%%
\begin{lemma}\label{POLSTA6}
Let $m=1$ and assume that \eqref{CONDITION-H} holds. The solution $(u,v,y,z,w,\eta)\in D(\mathcal{A}_1)$ of \eqref{1POL-eq1}-\eqref{1POL-eq6} satisfies  the following estimation
\begin{equation}\label{Lemma4-EQ1}
\int_0^L\abs{\la y}^2dx=o(1).	
\end{equation}
\end{lemma}
\begin{proof}	
Inserting 	\eqref{POL-eq3} in \eqref{POL-eq4}, we get 
\begin{equation*}
-\la^2\mu y-\beta y_{xx}+\gamma \beta u_{xx}=\mu\frac{f^4}{\la^2}+i\mu\frac{f^3}{\la}.	
\end{equation*}
Multiplying the above equation by $-\overline{y}$ integrating by parts over $(0,L)$, we get 
\begin{equation}\label{Lemma4-EQ2}
\mu\int_0^L\abs{\la y}^2dx=-\beta\int_0^L\abs{y_x}^2dx-\gamma\beta\int_0^Lu_x\overline{y_x}dx-\int_0^L\left(\frac{\mu f^4}{\la^2}+i\frac{\mu f^3}{\la}\right)\overline{y}dx.	
\end{equation}
Using \eqref{1poly-yx}, \eqref{Lemma4-EQPOL1}, $\la y$ is uniformly bounded in $L^2(0,L)$ and the fact that $\|F^n\|_{\mathcal{H}}=o(1)$, we get 
\begin{equation*}
\left|\int_0^Lu_x\overline{y_x}dx\right|=o(1)\quad \text{and}\quad 	\left|\int_0^L\left(\frac{\mu f^4}{\la^2}+i\frac{\mu f^3}{\la}\right)\overline{y}dx\right|=\frac{o(1)}{\la^2}.
\end{equation*}
Inserting the above estimations in \eqref{Lemma4-EQ2} and using \eqref{1poly-yx}, we get \eqref{Lemma4-EQ1}. The proof is thus completed. 
\end{proof}
\\

\noindent \textbf{Proof of Theorem \ref{POLY-STA-m=1}.}  Using Lemmas \ref{EST1-POL1}-\ref{POLSTA6}, we obtain that $\|U^n\|_{\mathcal{H}}=o(1)$, which contradicts $\|U^n\|_{\mathcal{H}}=1$. Thus,  \eqref{COND-POL2} holds true. The proof has been completed. \xqed{$\square$}
\section*{Conclusion}
\noindent In this work, we studied the  decay rate for one dimensional piezoelectric beams with magnetic effect and heat equation with memory, where the hereditary heat conduction is due to  Coleman-Gurtin law or Gurtin-Pipkin Law. The exponential stability is obtained when the hereditary heat conduction is of Coleman-Pikin type. Further, we show the polynomial stability of type $t^{-1}$ when the heat conduction law is of Gurtin-Pipkin law. The table below summarizes the results of the paper: 
$$
\begin{tabular}{|c|c|}
\hline
Systems & Stability\\
\hline
Piezoelectric with Coleman-Gurtin law \eqref{Pizo}& Exponential Stability\\
\hline 
Piezoelectric with Fourier law  \eqref{Pizo-F}& Exponential Stability\\
\hline 
Piezoelectric with Gurtin-Pipkin law \eqref{Pizo-GP}& Polynomial Stability of order $t^{-1}$\\
\hline
\end{tabular}
$$
We conjecture that the polynomial energy decay rate obtained in Theorem \ref{POLY-STA-m=1} is optimal. The idea of the proof is to find a sequence $(\lambda_n)_n\subset \mathbb{R}^{\ast}_+$ with $\abs{\la_n}\to +\infty$ and a sequence of vectors $(U_n)_n\subseteq D(\mathcal{A}_1)$  such that $(i\la_n I-\mathcal{A}_1)U_n=F_n$ is bounded in $\mathcal{H}$ and 
$$\lim_{n\to +\infty}\la_n^{-2+\varepsilon}\|U_n\|_{\mathcal{H}}=\infty.$$
 (See Theorem 3.1 in \cite{AIW-ACAP2021} and Theorem 5.1 in \cite{ABNWmemory}). Depending on  the boundary conditions, this approach and the construction of the vector $(U_n)$ is not feasible and the problem is still an open problem.  	 
%%%%%%%%%%%%%%%%%%%%%%%%%%%%%%%%%%%%%%%%%%%%%%%%%%%%%%%%%%%%%%%%%%%%%%%%%%%%%%
%%%%%%%%%%%%%%%%%%%%%%%%%%%%%%%%%%%%%%%%%%%%%%%%%%%%%%%%%%%%%%%%%%%%%%%%%%%%%%
%%%%%%%%%%%%                         Appendix
%%%%%%%%%%%%%%%%%%%%%%%%%%%%%%%%%%%%%%%%%%%%%%%%%%%%%%%%%%%%%%%%%%%%%%%%%%%%%%
%%%%%%%%%%%%%%%%%%%%%%%%%%%%%%%%%%%%%%%%%%%%%%%%%%%%%%%%%%%%%%%%%%%%%%%%%%%%%
\appendix
\section{Some notions and  stability theorems}
\noindent In order to make this paper more self-contained, we recall in this short appendix some notions and stability results used in this work. 
\begin{defi}\label{App-Definition-A.1}{\rm
Assume that $A$ is the generator of $C_0-$semigroup of contractions $\left(e^{tA}\right)_{t\geq0}$ on a Hilbert space $H$. The $C_0-$semigroup $\left(e^{tA}\right)_{t\geq0}$ is said to be 
\begin{enumerate}
\item[$(1)$] Strongly stable if 
$$
\lim_{t\to +\infty} \|e^{tA}x_0\|_H=0,\quad \forall\, x_0\in H.
$$
\item[$(2)$] Exponentially (or uniformly) stable if there exists two positive constants $M$ and $\varepsilon$ such that 
$$
\|e^{tA}x_0\|_{H}\leq Me^{-\varepsilon t}\|x_0\|_{H},\quad \forall\, t>0,\ \forall\, x_0\in H.
$$
\item[$(3)$] Polynomially stable if there exists two positive constants $C$ and $\alpha$ such that 
$$
\|e^{tA}x_0\|_{H}\leq Ct^{-\alpha}\|A x_0\|_{H},\quad \forall\, t>0,\ \forall\, x_0\in D(A).
$$
\xqed{$\square$}
\end{enumerate}}
\end{defi}
%%%%%%%%%%%%%%%%%%%%%%%%%%%%%%%%%%%%%%%%%%%%%%%%%%%%%%%%%%%%%%%%%%%%%%%%%%%%%%
\noindent For proving the strong stability of the $C_0$-semigroup $\left(e^{tA}\right)_{t\geq0}$, we will recall the result obtained by Arendt and Batty in \cite{Arendt01}. 
\begin{Theorem}[Arendt and Batty in \cite{Arendt01}]\label{App-Theorem-A.2}{\rm
{Assume that $A$ is the generator of a C$_0-$semigroup of contractions $\left(e^{tA}\right)_{t\geq0}$  on a Hilbert space $H$. If $A$ has no pure imaginary eigenvalues and  $\sigma\left(A\right)\cap i\mathbb{R}$ is countable,
where $\sigma\left(A\right)$ denotes the spectrum of $A$, then the $C_0$-semigroup $\left(e^{tA}\right)_{t\geq0}$  is strongly stable.}\xqed{$\square$}}
\end{Theorem}
\noindent There exist a second classical method based on Arendt and Batty theorem and the contradiction argument  (see  page 25 in \cite{LiuZheng01}).
%%%%%%%%%%%%%%%%%%%%%%%%%%%%%%%%%%%%%%%%%%%%%%%%%%
                % Lemma %
%%%%%%%%%%%%%%%%%%%%%%%%%%%%%%%%%%%%%%%%%%%%%%%%%%
\begin{Remark}\label{App-Lemma-A.3}{\rm
 Assume that the unbounded linear operator $A:D(A)\subset H \longmapsto H$ is the generator of a C$_0-$semigroup of contractions $\left(e^{tA}\right)_{t\geq0}$  on a Hilbert space $H$ and suppose that $0\in \rho(A).$ According  to  (page 25 in \cite{LiuZheng01}, see also \cite{ABW-CPAA2021}), in order to prove that \begin{equation}\label{A1}
 \displaystyle	i\R \equiv \left\{ i\la\  | \ \ \la\in\R \right\} \subseteq \rho(A),
\end{equation} we need the following steps:
\begin{enumerate}
\item[(i)] It follows from the fact that $0\in \rho (A)$ and the contraction mapping theorem that for any real number $\la$ with $|\la|<\|A^{-1} \|^{-1} $, the operator $i\la I-A =A(i\la A^{-1} -I)$ is invertible. Furthermore, $\|(i\la I-A)^{-1} \|$ is a continuous function of $\la $ in the interval $\left(-\|A^{-1} \|^{-1} ,\|A^{-1} \|^{-1} \right) $.\\
\item[(ii)] If $\sup \left\{\|(i\la I-A)^{-1}\|\ | \ |\la|<\|A^{-1}\|^{-1} \right\}=M < \infty$, then by the contraction mapping theorem, the operator $i\la I-A =(i\la_0 I-A)(I+i(\la-\la_0 )(i\la_0 I-A)^{-1} )$ with $|\la_0 |<\|A^{-1}\|^{-1}$ is invertible for $|\la -\la_0 | <M^{-1}$. It turns out that by choosing $|\la_0 | $ as close to $\|A^{-1}\|^{-1}$ as we can, we conclude that $\left\{ \la \ | \ |\la|<\|A^{-1}\|^{-1}+M^{-1}\right \} \subset \rho (A)$ and $\|(i\la I -A)^{-1}\|$ is a continuous function of $\la $  in the interval $\left(-\|A^{-1}\|^{-1}-M^{-1},\|A^{-1}\|^{-1}+M^{-1} \right).$\\ 
\item[(iii)] Thus it follows from the argument in (ii) that if \eqref{A1} is false, then there is $\omega \in\R $ with $\|A^{-1}\|^{-1}\leq |\omega| <\infty $ such that $\left\{i\la \ | \ |\la |<|\omega| \right\} \subset \rho(A)  $ and $\sup \left\{\|(i\la -A)^{-1}\| \ | \ |\la |<|\omega | \right\}=\infty$. It turns out that there exists a sequence $\left\{(\lambda_n,{U}_n)\right\}_{n\geq 1}\subset \mathbb{R}\times D\left(A\right),$ with $\lambda_n \to  \omega$ as $n\to\infty,$ $|\lambda_n|<|\omega|$ and $\left\|{U}_n\right\|_{H} = 1$, such that
%%%%%%%%%%%%%%%%%%%%Equation%%%%%%%%%%%%%%%%%%%%%%
\begin{equation*}
(i \lambda_n I - A){U}_n ={F}_n \to  0 \ \textrm{in} \ {H},\qquad\text{as }n\to\infty.
\end{equation*}
Then, we will prove \eqref{A1}  by finding a contradiction with $\left\|{U}_n\right\|_{H} = 1$ such as $\left\| {U}_n\right\|_{H}  \to 0 .$\xqed{$\square$}.\\ \end{enumerate}}
\end{Remark}

\noindent We now recall the following standard result which is stated in a comparable way \cite{Huang01, pruss01}   for part (1) and   \cite{Borichev01} (see also \cite{Batty01}, \cite{RaoLiu01} and \cite{Rozendaal2017}) for part (2) .
%%%%%%%%%%%%%%%%%%%%%%%%%%%%%%%%%%%%%%%%%%%%%%%%%%
% Theorem %
%%%%%%%%%%%%%%%%%%%%%%%%%%%%%%%%%%%%%%%%%%%%%%%%%%
\begin{Theorem}\label{bt}{\rm
	Assume that $A$ is the generator of a strongly continuous semigroup of contractions $\left(e^{tA}\right)_{t\geq0}$  on $H$.   If   $ i\mathbb{R}\subset \rho(A)$. Then;\\
\begin{enumerate}
\item[(1)] The semigroup $e^{tA}$ is exponentially stable if and only if 
$$
\limsup_{\la\in \mathbb{R},\la\to \infty}\|(i\la I-A)^{-1}\|<\infty.
$$
\item[(2)]	The semigroup $e^{tA}$ is polynomially stable of order $\ell>0$ if and only if 
$$
\limsup_{\la\in \mathbb{R},\la\to \infty}\abs{\la}^{-\frac{1}{\ell}}\|(i\la I-A)^{-1}\|<\infty.
$$
\end{enumerate}

}
\end{Theorem}
%%%%%%%%%%%%%%%%%%%%%%%%%%%%%%%%%%%%%%%%%%%%%%%%%%
%%%%%%%%%%%%%%%%%%%%%%%%%%%%%%%%%%%%%%%%%%%%%%%%%%
%%%%%%%%%%%%%%%%%%%%%%%%%%%%%%%%%%%%%%%%%%%%%%%%%%
 %  References %
%%%%%%%%%%%%%%%%%%%%%%%%%%%%%%%%%%%%%%%%%%%%%%%%%%
%%%%%%%%%%%%%%%%%%%%%%%%%%%%%%%%%%%%%%%%%%%%%%%%%%	
%%%%%%%%%%%%%%%%%%%%%%%%%%%%%%%%%%%%%%%%%%%%%%%%%%

\end{document}